\documentclass[a4paper, 11pt]{amsart}
\usepackage{amsmath, amssymb, amsthm, mathrsfs}
\usepackage[mathscr]{eucal}
\usepackage[all]{xy}
\usepackage[dvipdfmx]{graphicx}
\usepackage{comment}
\usepackage{mathtools}
\usepackage{bm} 

\DeclareSymbolFont{stmry}{U}{stmry}{m}{n}
\DeclareMathDelimiter\llbracket{\mathopen}{stmry}{"4A}{stmry}{"71}
\DeclareMathDelimiter\rrbracket{\mathclose}{stmry}{"4B}{stmry}{"79}

\theoremstyle{plain}
 \newtheorem{thm}{Theorem}[section]
 \newtheorem{lem}[thm]{Lemma}
 \newtheorem{cor}[thm]{Corollary}
 \newtheorem{prop}[thm]{Proposition}
 \newtheorem{claim}[thm]{Claim}
\theoremstyle{definition}
  \newtheorem{defn}[thm]{Definition}
  
  \newtheorem{notation}[thm]{Notation}

\theoremstyle{remark}
  \newtheorem{rem}[thm]{Remark}
  \newtheorem{ex}[thm]{Example}

\newcommand{\R}{\mathbb{R}}

\newcommand{\N}{\mathbb{N}}
\newcommand{\Z}{\mathbb{Z}}
\newcommand{\calr}{\mathcal{R}}

\newcommand{\calg}{\mathcal{G}}
\newcommand{\calh}{\mathcal{H}}
\newcommand{\cals}{\mathcal{S}}

\newcommand{\calq}{\mathcal{Q}}
\newcommand{\cale}{\mathcal{E}}
\newcommand{\caln}{\mathcal{N}}
\newcommand{\calm}{\mathcal{M}}

\newcommand{\bs}{\mathrm{BS}}

\newcommand{\ve}{\varepsilon}

\newcommand{\aut}{\mathrm{Aut}}
\newcommand{\End}{\mathrm{End}}
\newcommand{\Q}{\mathbb{Q}}

\newcommand{\ci}[2]{\cite[#1]{#2}}
\renewcommand{\c}{\curvearrowright}

  \makeatletter
  \@addtoreset{equation}{section}
  \makeatother

\begin{document}

\title{On treeings arising from HNN extensions}
\author{Yoshikata Kida}
\address{Graduate School of Mathematical Sciences, the University of Tokyo, Komaba, Tokyo 153-8914, Japan}
\email{kida@ms.u-tokyo.ac.jp}
\date{September 14, 2023}
\thanks{The author was supported by JSPS Grant-in-Aid for Scientific Research, 17K05268.}

\begin{abstract}
For certain HNN extensions including Baumslag-Solitar groups, a treeing is constructed from their certain probability-measure-preserving actions.
This is a treeing of a quotient groupoid of the translation groupoid associated with their actions.
As its application, for some of those HNN extensions, we show that the kernel of the modular homomorphism is measure equivalent to the direct product of the free group of infinite rank and $\Z$.
\end{abstract}

\maketitle


\section{Introduction}\label{sec-intro}

Treeings are, so to speak, free generating systems for countable measured equivalence relations, introduced by Adams \cite{adams1}, \cite{adams2}.
One of the most remarkable results on treeings is Gaboriau's theorem saying that treeings realize the cost of an equivalence relation \cite{gaboriau}.
As its consequence, it follows that free groups of different rank cannot admit free p.m.p.\ (probability-measure-preserving) actions which are orbit equivalent.

In this paper, we construct a treeing of certain discrete measured groupoids arising as follows:
Let $G$ be the HNN extension
\[G=\langle \, E, \, t\mid \forall a\in E_-\ \ tat^{-1}=\tau(a)\, \rangle,\]
where $E$ is a countable group and $\tau \colon E_-\to E_+$ is an isomorphism between finite-index normal subgroups $E_-$, $E_+$ of $E$.
Given a p.m.p.\ action $G\c (X, \mu)$ on a standard probability space satisfying a certain periodicity condition for its restriction to $E$, we have the translation groupoid $(X\rtimes G, \mu)$, and remarkably the subgroupoid $X\rtimes E$ is normal in it in the sense of Feldman-Sutherland-Zimmer \cite{fsz}, while $E$ is not normal in $G$ except for the trivial case.
We then obtain the quotient groupoid $(\calq, \zeta)$ of $(X\rtimes G, \mu)$ by $X\rtimes E$.
The first aim of this paper is to construct a treeing of this groupoid $(\calq, \zeta)$ and to study its basic properties.
Almost every component of the treeing we obtain is identified with the Bass-Serre tree associated with the HNN decomposition of $G$.

We are mainly interested in the case when both $p\coloneqq [E: E_-]$ and $q\coloneqq [E: E_+]$ are more than $1$ and $p\neq q$.
We should note that in this case $(\calq, \zeta)$ is not p.m.p.\ (or rather, is of type $\mathrm{III}$).
Originally this groupoid $(\calq, \zeta)$ appears in the author's study of Baumslag-Solitar groups \cite{kida-bs} and has been curious to the author.
We hope the investigation of this paper will lead to understanding this mysterious groupoid.

In Gaboriau's paper \cite{gaboriau}, several techniques to handle treeings are developed.
One of them is the induction.
This enables us to construct a treeing of the restriction of $\calr$ when a treeing of an equivalence relation $\calr$ is given.
Applying this induction technique to the Maharam extension of $(\calq, \zeta)$, together with Hjorth's result on cost \cite{hjorth}, we show Theorem \ref{thm-main} below.
This is the second aim of this paper.

In this paper we mean by a \textit{p.m.p.}\ action a measure-preserving action on a standard probability space.
Let us say that two countable groups are \textit{orbit equivalent} if some free p.m.p.\ action of one of the two groups is orbit equivalent to some free p.m.p.\ action of the other group.
We refer to e.g., \cite{furman} for basic terminology and results on orbit equivalence.

\begin{thm}\label{thm-main}
Let $G$ be the above HNN extension.
Suppose that $E$ is finitely generated, free abelian, and suppose also $p>1$, $q>1$ and $p\neq q$.
Let $\bm{m}\colon G\to \Q^*_+$ be the modular homomorphism associated to $E$ defined by
\[\bm{m}(g)=\frac{[E: E\cap gEg^{-1}]}{[gEg^{-1}: E\cap gEg^{-1}]}\]
for $g\in G$ (see Subsection \ref{subsec-rn}). 
Then $\ker \bm{m}$ is orbit equivalent to $F_\infty \times \Z$, where $F_\infty$ is the free group of countably infinite rank.
\end{thm}

The group $\ker \bm{m}$ in this theorem is written as the iterated amalgamated free product of the bi-infinite chain of copies of $E$:
For each $n\in \Z$, prepare a copy $E_n=E$.
Take the free product of all $E_n$ with $n\in \Z$, and for every $n\in \Z$, identify each $a\in E_-<E_n$ with $\tau(a)\in E_+<E_{n-1}$.
Then the resulting group is isomorphic to $\ker \bm{m}$ (via identifying $E_n$ with $t^nEt^{-n}$).

We expect that the same conclusion holds for other kinds of iterated amalgamated free products (e.g., the kernel of the modular homomorphism of a generalized Baumslag-Solitar group).
This generalization is not straightforward because our construction of the treeing of the groupoid $(\calq, \zeta)$ depends on transitivity of the action of $G$ on the vertex set of the Bass-Serre tree.
However we expect that our construction of the treeing can be extended and that the groups dealt with in this paper are the simplest ones that our construction is available for.
We hope this extension will be carried out in future. 

Finally let us mention the author's work \cite{kida-bs} on orbit equivalence between Baumslag-Solitar groups $\bs(p, q)=\langle \, a,\, t\mid ta^pt^{-1}=a^q\, \rangle$ for nonzero integers $p$, $q$.
Among others, it is shown that if two Baumslag-Solitar groups $\bs(p, q)$ and $\bs(p', q')$ with $2\leq |p|< |q|$ and $2\leq |p'|< |q'|$ are orbit equivalent, then the kernels of their modular homomorphisms are orbit equivalent \cite[Corollary 7.5]{kida-bs}.
This was our motivation to study the group $\ker \bm{m}$ in Theorem \ref{thm-main}.
However it turns out from Theorem \ref{thm-main} that this result in \cite{kida-bs} does not contribute to the orbit equivalence classification of Baumslag-Solitar groups.
We refer the reader to \cite{hh}, \cite{hr}, \cite{kida-stab}, \cite{kida-mcmr}, \cite{mee} and \cite{td} for other works on orbit equivalence classes of p.m.p.\ actions of Baumslag-Solitar groups and related groups.

The paper is organized as follows.
In Section \ref{sec-normal}, we start with preliminaries on discrete measured groupoids, and discuss normal subgroupoids and quotient groupoids.
We collect criteria for subgroupoids to be normal and for the quotient groupoid to be p.m.p.
In Section \ref{sec-treeing}, we introduce a treeing of a discrete Borel groupoid and its induction.
In Section \ref{sec-HNN}, for the above HNN extension $G$ and a certain p.m.p.\ action $G\c (X, \mu)$, we construct a treeing of the quotient groupoid $(\calq, \zeta)$ of $(X\rtimes G, \mu)$ by $X\rtimes E$, and apply the induction to the Maharam extension of $(\calq, \zeta)$.
In Section \ref{sec-splitting}, after proving a general criterion for splitting of groupoid-extensions, we prove Theorem \ref{thm-main}.


\section{Normal subgroupoids and quotient groupoids}\label{sec-normal}

Throughout the paper, all measurable spaces are standard Borel spaces.
We use the adjectives ``Borel'' and ``measurable'' interchangeably. 
In the context of measure spaces, unless otherwise mentioned, all relations among Borel sets and maps are understood to hold up to null sets.

\subsection{Preliminaries on groupoids}\label{subsec-pre}

Let $\calg$ be a groupoid.
Let $\calg^0\subset \calg$ denote the set of units and $r, s\colon \calg \to \calg^0$ denote the range and source maps, respectively.
We employ the notation in \cite[Section 3.1]{btd}.
Given two subsets $A, B\subset \calg$, we define their product as $AB=\{ \, gh \mid g\in A,\, h\in B,\, s(g)=r(h)\, \}$.
This product is associative and hence the product $A_1\cdots A_n$ of finitely many subsets $A_1,\ldots, A_n\subset \calg$ makes sense.
If $A, B\subset \calg^0$, then $A\calg =r^{-1}(A)$, $\calg B=s^{-1}(B)$ and $A\calg B=r^{-1}(A)\cap s^{-1}(B)$.
For $g\in \calg$ and $A\subset \calg$, we write $gA$ for $\{ g\}A$ and write $Ag$ for $A\{ g\}$.
Therefore for $x\in \calg^0$, we have $x\calg =r^{-1}(x)$ and $\calg x=s^{-1}(x)$.
For $A\subset \calg$, we write $A^{-1}=\{ \, g^{-1}\mid g\in A\, \}$.

Let $A\subset \calg^0$ be a subset.
We write $\calg|_A= A\calg A$, which is a groupoid whose unit space is $A$ if $A$ is nonempty.
We write $[A]_\calg$ for the saturation of $A$, i.e., $[A]_\calg = s(A\calg)$.
If $[A]_\calg =A$, then $A$ is called \textit{$\calg$-invariant}.
We mean by an \textit{$r$-section} (resp.\ an \textit{$s$-section}) of $\calg$  a subset $\phi \subset \calg$ such that the range (resp.\ source) map is injective on $\phi$.
We mean by a \textit{bisection} of $\calg$ a subset $\phi \subset \calg$ which is an $r$-section of $\calg$ and is also an $s$-section of $\calg$.
If $\phi$ and $\psi$ are $r$-sections of $\calg$, then $\phi \psi$ is also an $r$-section of $\calg$.
For an $r$-section $\phi$ of $\calg$ and $x\in r(\phi)$, the set $x\phi$ consists of a single element, and by abuse of notation we denote this element by $x\phi$.
Similarly for an $s$-section $\phi$ of $\calg$ and $y\in s(\phi)$, we denote the single element of $\phi y$ by the same symbol $\phi y$.
For an $r$-section $\phi$ of $\calg$, we define the map $\bar{\phi}\colon r(\phi)\to s(\phi)$ by $\bar{\phi}(x)=s(x\phi)$ for $x\in r(\phi)$.

We mean by a \textit{discrete Borel groupoid} a groupoid $\calg$ such that $x\calg$ is countable for every $x\in \calg^0$, $\calg$ is a standard Borel space, $\calg^0$ is a Borel subset of $\calg$, and all of the maps $r$, $s$ and the multiplication and inverse maps of $\calg$ are Borel.

For $x\in \calg^0$, let $c_x^r$ and $c_x^s$ be the counting measures on $x\calg$ and $\calg x$, respectively.
We mean by a \textit{discrete measured groupoid} a pair $(\calg, \mu)$ of a discrete Borel groupoid $\calg$ and a probability measure $\mu$ on $\calg^0$ such that the two measures on $\calg$,
\begin{equation}\label{eq-mu}
\mu_\calg^r\coloneqq \int_{\calg^0}c_x^r\, d\mu(x)\quad \text{and} \quad \mu_\calg^s\coloneqq \int_{\calg^0}c_x^s\, d\mu(x),
\end{equation}
are equivalent.
Note that $\mu_\calg^r$ and $\mu_\calg^s$ are equivalent if and only if for every Borel bisection $\phi$ of $\calg$, the two measures on $s(\phi)$, $\bar{\phi}_*(\mu|_{r(\phi)})$ and $\mu|_{s(\phi)}$, are equivalent.
In this case, further for every Borel $r$-section $\phi$ of $\calg$, the two measures on $s(\phi)$, $\bar{\phi}_*(\mu|_{r(\phi)})$ and $\mu|_{s(\phi)}$, are equivalent.
Given a discrete measured groupoid $(\calg, \mu)$, we endow $\calg$ with the measure $\mu_\calg^r$ unless otherwise mentioned.

Let $\mu$ be a $\sigma$-finite measure on $\calg^0$ and define the two measures $\mu_\calg^r$, $\mu_\calg^s$ by the same formula as (\ref{eq-mu}).
The equation $\mu_\calg^r=\mu_\calg^s$ holds if and only if for every Borel bisection $\phi$, the equation $\bar{\phi}_*(\mu|_{r(\phi)})=\mu|_{s(\phi)}$ holds.
In this case, we say that $\calg$ \textit{preserves} $\mu$.
If moreover $\mu$ is a probability measure, then we call the pair $(\calg, \mu)$ a \textit{discrete p.m.p.\ groupoid}.

Let $(\calg, \mu)$ be a discrete measured groupoid.
We denote by $\llbracket (\calg, \mu)\rrbracket$ the set of Borel bisections of $\calg$, where two of them are identified if they coincide up to null sets.
We write $\llbracket \calg \rrbracket$ for $\llbracket (\calg, \mu)\rrbracket$ if $\mu$ is understood from the context.
We denote by $[ (\calg, \mu)]$ (or $[\calg]$ simply) the set of all $\phi \in \llbracket (\calg, \mu)\rrbracket$ such that $r(\phi)=s(\phi)=\calg^0$.

Let $\cals <\calg$ be a Borel subgroupoid (where we always assume a Borel subgroupoid of $\calg$ to have the same unit space as that of $\calg$).
For $x\in \calg^0$, we define an equivalence relation $\sim_\cals$ on $x\calg$ by saying that for $g, h\in x\calg$, $g\sim_\cals h$ if and only if $h^{-1}g\in \cals$.
The function assigning to $x\in \calg^0$ the cardinality of the $\sim_\cals$-equivalence classes in $x\calg$ is Borel and $\calg$-invariant \cite[Section 1]{fsz} (or \cite[Lemma 3.5 (i)]{kida-bs}).
This function is called the \text{index function} for the pair $\cals <\calg$, and if it is constant almost everywhere, then the constant is called the \textit{index} of $\cals$ in $\calg$.

A countable family $(\phi_n)$ of Borel $r$-sections $\phi_n$ of $\calg$ with $r(\phi_n)=\calg^0$ is called a \textit{family of choice functions} for the pair $\cals <\calg$ if for almost every $x\in \calg^0$, for every $g\in x\calg$, there exists a unique $n$ such that $g\sim_\cals x\phi_n$.
We define $\End_\calg(\cals)$ as the set of Borel $r$-sections $\phi$ of $\calg$ such that $r(\phi)=\calg^0$ and the associated conjugation
\[V_\phi \colon \calg \to \calg,\quad V_\phi(\gamma)=\phi^{-1}\gamma \phi\]
satisfies $V_\phi(\cals)\subset \cals$.
We call $\cals$ \textit{normal} in $\calg$ (or in $(\calg, \mu)$ when $\mu$ should be made explicit) if there exists a family of choice functions for $\cals <\calg$ which consist of elements in $\End_\calg(\cals)$.
In this case, we write $\cals \vartriangleleft \calg$ (\cite[Section 2]{fsz}).

\begin{ex}
Let $(X, \mu)$ be a standard probability space and let $\calr$ be a countable Borel equivalence relation on $X$ which preserves the class of $\mu$.
Then $(\calr, \mu)$ is a discrete measured groupoid with respect to the range map $r(x, y)= x$, the source map $s(x, y)= y$, the product $(x, y)(y, z)=(x, z)$, and the inverse $(x, y)^{-1}=(y, x)$.
Each Borel $r$-section $\phi$ of $\calr$ may be identified with the map $\bar{\phi} \colon r(\phi)\to s(\phi)$.

Let $\cals < \calr$ be a Borel subequivalence relation.
Then for a Borel $r$-section $\phi$ of $\calr$ with $r(\phi)=X$, $\phi$ belongs to $\End_\calr(\cals)$ if and only if for every $(x, y) \in \cals$, we have $\phi^{-1}(x, y)\phi \in \cals$, i.e., $(\bar{\phi}(x), \bar{\phi}(y))\in \cals$.
\end{ex}

\begin{ex}
Let $G\c X$ be a Borel action of a countable group on a standard Borel space.
Then the set $X\times G$ admits the structure of a discrete Borel groupoid with the range map $(x, g)\mapsto x$, the source map $(x, g)\mapsto g^{-1}x$, the product $(x, g)(g^{-1}x, h)=(x, gh)$, and the inverse $(x, g)^{-1}=(g^{-1}x, g^{-1})$.
This groupoid is denoted by $X\rtimes G$.
If the action $G\c X$ preserves the class of a probability measure $\mu$ on $X$, then $(X\rtimes G, \mu)$ is a discrete measured groupoid.

Let $H$ be a normal subgroup of $G$.
Then $\calh \coloneqq X\rtimes H$ is normal in $\calg \coloneqq X\rtimes G$.
Indeed if we choose $s_n\in G$ such that $G=\bigsqcup_n s_nH$ and define a bisection $\phi_n$ of $\calg$ by $\phi_n=X\times \{ s_n\}$, then $\phi_n\in \End_{\calg}(\calh)$ and $(\phi_n)$ is a family of choice functions for $\calh <\calg$.
\end{ex}

\begin{lem}\label{lem-s-inv}
Let $(\calg, \mu)$ be a discrete measured groupoid and let $\cals  <\calg$ be a Borel subgroupoid.
Then the following assertions hold:
\begin{enumerate}
\item[(i)] For all $\phi, \psi \in \End_\calg(\cals)$, the set of all $x\in \calg^0$ with $x\phi \sim_\cals x\psi$ is $\cals$-invariant. 
\item[(ii)] If $\phi, \psi \in \End_\calg(\cals)$, then $\phi \psi \in \End_\calg(\cals)$.
\end{enumerate}
\end{lem}

\begin{proof}
For assertion (i), pick $\gamma \in \cals$.
Then $\phi^{-1}r(\gamma)\psi=(\phi^{-1}\gamma \phi)\phi^{-1}s(\gamma)\psi(\psi^{-1}\gamma^{-1}\psi)$.
Since $\phi^{-1}\gamma \phi$ and $\psi^{-1}\gamma^{-1}\psi$ belong to $\cals$, assertion (i) follows.

For assertion (ii), pick $\gamma \in \cals$.
Then $(\phi \psi)^{-1}\gamma (\phi \psi)=\psi^{-1}(\phi^{-1}\gamma \phi )\psi$.
It follows from $\phi \in \End_\calg(\cals)$ that $\phi^{-1}\gamma \phi \in \cals$, and follows from $\psi \in \End_\calg(\cals)$ that $\psi^{-1}(\phi^{-1}\gamma \phi )\psi \in \cals$.
Assertion (ii) follows.
\end{proof}


\subsection{Construction of quotient groupoids}

Given a discrete measured groupoid and its normal subgroupoid, we construct the quotient groupoid.
For principal groupoids (i.e., equivalence relations), this is done by Feldman-Sutherland-Zimmer \cite[Theorem 2.2]{fsz}.
Our construction below is a verbatim translation of theirs.
For the reader's convenience, we give its details here.

\begin{thm}\label{thm-quotient}
Let $(\calg, \mu)$ be a discrete measured groupoid and let $\cals \vartriangleleft \calg$ be a normal Borel subgroupoid.
Then there exist a discrete measured groupoid $(\calq, \zeta)$ and a Borel homomorphism $\theta \colon \calg \to \calq$ with $\theta_*\mu =\zeta$ such that
\begin{enumerate}
\item[(1)] $\ker \theta =\cals$,
\item[(2)] $\theta$ is \textit{class-surjective}, i.e., for almost every $x\in \calg^0$, the map $x\calg \to \theta(x)\calq$ defined as the restriction of $\theta$ is surjective, and
\item[(3)] the following universal property holds: 
If $\calq'$ is a discrete Borel groupoid and $\theta'\colon \calg \to \calq'$ is a Borel homomorphism such that $\cals <\ker \theta'$, then there exists a Borel homomorphism $\tau \colon \calq \to \calq'$ such that $\theta' =\tau \circ \theta$.
\end{enumerate}
\end{thm}

The pair $(\calq, \zeta)$ in this theorem is called the \textit{quotient groupoid} of $(\calg, \mu)$ by $\cals$, and $\calq$ is denoted by $\calg /\cals$.

\begin{proof}[Proof of Theorem \ref{thm-quotient}]
We may assume that the index function for $\cals <\calg$ is constant, and let $J$ be a countable set whose cardinality is the index of $\cals$ in $\calg$.
Since $\cals \vartriangleleft \calg$, we have a family $(\phi_j)_{j\in J}$ of choice functions for $\cals <\calg$ such that $\phi_j\in \End_\calg(\cals)$ for all $j\in J$.
Let $\theta \colon (\calg^0, \mu)\to (Z, \zeta)$ be the ergodic decomposition map for $\cals$ with $\theta_*\mu =\zeta$.
We will make the set $Z\times J$ into a groupoid whose unit space is identified with $Z$.

Let $x\in \calg^0$ and $j, k\in J$.
Then there exists a unique $l\in J$ such that $x\phi_j\phi_k\sim_\cals x\phi_l$ since $(\phi_j)_{j\in J}$ is a family of choice functions for $\cals <\calg$.
This $l$ is denoted by $j\ast_xk$.

\begin{lem}\label{lem-ast-s-inv}
Fix $j, k\in J$.
Then the map $\calg^0\ni x\mapsto j\ast_xk$ is $\cals$-invariant.
\end{lem}

\begin{proof}
Pick $\gamma \in \cals$ and put $x=s(\gamma)$, $y=r(\gamma)$ and $l=j\ast_xk$.
Then
\[(\phi_j\phi_k)^{-1}\gamma (\phi_j\phi_k)=(\phi_j\phi_k)^{-1}\gamma (x\phi_l)(x\phi_l)^{-1}(x\phi_j\phi_k).\]
The left hand side belongs to $\cals$ (because $\phi_j\phi_k\in \End_\calg(\cals)$), and $(x\phi_l)^{-1}(x\phi_j\phi_k)$ belongs to $\cals$ by the definition of $l$.
Therefore $(\phi_j\phi_k)^{-1}\gamma \phi_l\in \cals$, which is equal to
\[(y\phi_j\phi_k)^{-1}(y\phi_l)(\phi_l^{-1}\gamma \phi_l).\]
Since $\phi_l^{-1}\gamma \phi_l\in \cals$, we have $(y\phi_j\phi_k)^{-1}(y\phi_l)\in \cals$.
Thus $l=j\ast_y k$.
\end{proof}

By Lemma \ref{lem-ast-s-inv}, for all $j, k\in J$, we obtain a Borel map $Z\to J$, $z\mapsto j\ast_zk$ such that $j\ast_{\theta(x)}k=j\ast_xk$ for almost every $x\in \calg^0$.

\begin{lem}\label{lem-Phi-j}
Fix $j\in J$.
Then the map $\calg^0\to Z$, $x\mapsto \theta(\bar{\phi}_j(x))$ is $\cals$-invariant.
\end{lem}

\begin{proof}
For every $\gamma \in \cals$, we have $\phi_j^{-1}\gamma \phi_j\in \cals$, and its range and source are $\bar{\phi}_j(r(\gamma))$ and $\bar{\phi}_j(s(\gamma))$, respectively.
Hence $\theta(\bar{\phi}_j(r(\gamma)))=\theta(\bar{\phi}_j(s(\gamma)))$.
\end{proof}

We thus obtain a Borel map $\Phi_j\colon Z\to Z$ such that
\[\Phi_j(\theta(x))=\theta(\bar{\phi}_j(x))=\theta(s(x\phi_j))\]
for almost every $x\in \calg^0$.

We set $\calq =Z\times J$.
We define the range and source maps $r, s\colon \calq \to Z$ by 
\[r(z, j)=z,\quad s(z, j)=\Phi_j(z),\quad (z, j)\in \calq.\]
These maps are Borel.
We define the product on $\calq$ by
\[(z, j)(\Phi_j(z), k)=(z, j\ast_zk).\]
This is Borel since the map $z\mapsto j\ast_zk$ is Borel for all fixed $j, k\in J$. 
We check the source of the right hand side is equal to the source of $(\Phi_j(z), k)$, that is, we check the equation
\[\Phi_{j\ast_zk}(z)=\Phi_k(\Phi_j(z)).\]
We put $l=j\ast_zk$ and $w=\Phi_j(z)$.
For $x, y\in \calg^0$ with $\theta(x)=z$ and $\theta(y)=w$, we have
\begin{align*}
\Phi_l(z)=\theta(s(x\phi_l))=\theta(s(x\phi_j\phi_k))=\theta(s(s(x\phi_j)\phi_k))=\theta(s(y\phi_k))=\Phi_k(w),
\end{align*}
which is the desired equation.
Here the second equation follows from the definition of $l$.
The fourth equation holds because $\theta(s(x\phi_j))=\Phi_j(z)=w=\theta(y)$ and hence $s(x\phi_j)$ and $y$ belong to the same $\theta$-fiber, and the map $\calg^0\ni u\mapsto \theta(s(u\phi_k))$ is $\cals$-invariant by Lemma \ref{lem-Phi-j}.

\begin{lem}
The product on $\calq$ is associative. 
\end{lem}

\begin{proof}
For each $x\in \calg^0$, write $\bar{x}=\theta(x)$ for simplicity.
Pick $x, y, z\in \calg^0$ and $j, k, l\in J$, and suppose that the two products $(\bar{x}, j)(\bar{y}, k)$, $(\bar{y}, k)(\bar{z}, l)$ are defined, i.e., $\bar{y}=\Phi_j(\bar{x})$ and $\bar{z}=\Phi_k(\bar{y})$.
We verify the equation
\[[(\bar{x}, j)(\bar{y}, k)](\bar{z}, l)=(\bar{x}, j)[(\bar{y}, k)(\bar{z}, l)].\]
The left hand side is
\[(\bar{x}, j\ast_{\bar{x}}k)(\bar{z}, l)=(\bar{x}, (j\ast_{\bar{x}}k)\ast_{\bar{x}}l),\]
and the right hand side is
\[(\bar{x}, j)(\bar{y}, k\ast_{\bar{y}}l)=(\bar{x}, j\ast_{\bar{x}}(k\ast_{\bar{y}}l)).\]
Therefore it suffices to verify the equation
\begin{equation}\label{eq-ast}
(j\ast_{\bar{x}}k)\ast_{\bar{x}}l=j\ast_{\bar{x}}(k\ast_{\bar{y}}l).
\end{equation}

We put $m=j\ast_{\bar{x}}k$.
For $n\in J$, $n=m\ast_{\bar{x}}l$ if and only if $x\phi_n\sim_\cals x\phi_{m}\phi_l$.
We have
\[x\phi_m\phi_l\sim_\cals x\phi_j\phi_k\phi_l\]
because $x\phi_m\sim_\cals x\phi_j\phi_k$ and $\phi_l\in \End_\calg(\cals)$.
Regarding to the right hand side of equation (\ref{eq-ast}), we put $p=k\ast_{\bar{y}}l$.
For $n\in J$, $n=j\ast_{\bar{x}}p$ if and only if $x\phi_n\sim_\cals x\phi_j\phi_p$.
Putting $w=s(x\phi_j)$, we have
\[x\phi_j\phi_p\sim_\cals x\phi_j\phi_k\phi_l\]
because $\bar{y}=\Phi_j(\bar{x})=\bar{w}$ and hence $w\phi_p\sim_\cals w\phi_k\phi_l$.
Thus equation (\ref{eq-ast}) follows.
\end{proof}

\begin{lem}
For every $z\in Z$, there exists $u_z\in \calq$ such that $r(u_z)=s(u_z)=z$, $u_zg=g$ for all $g\in r^{-1}(z)$, and $hu_z=h$ for all $h\in s^{-1}(z)$.
\end{lem}

\begin{proof}
For every $x\in \calg^0$, there exists a unique $i_x\in J$ such that $x\phi_{i_x}\sim_\cals x$, i.e., $x\phi_{i_x}\in \cals$.
The map $x\mapsto i_x$ is Borel.
By Lemma \ref{lem-s-inv} (i), the map is $\cals$-invariant (since $\calg^0\in \End_\calg(\cals)$).
We obtain a Borel map $Z\to J$, $z\mapsto i_z$.

Fix $z\in Z$.
We show that $(z, i_z)$ is the unit of $\calq$ at $z$.
Choosing $x\in \calg^0$ with $\theta(x)=z$, we have
\[\Phi_{i_z}(z)=\Phi_{i_x}(\theta(x))=\theta(s(x\phi_{i_x}))=\theta(x)=z,\]
where the third equation holds because $x\phi_{i_x}\in \cals$.
It follows that the source of $(z, i_z)$ is $z$.

For every $j\in J$, we have $(z, i_z)(z, j)=(z, i_z\ast_zj)$ and $x\phi_{i_z}\phi_j\sim_\cals x\phi_j$ because $x\phi_{i_z}\in \cals$ and $\phi_j\in \End_\calg(\cals)$.
Therefore $i_z\ast_zj=j$.

Putting $w=\Phi_j(z)$, we have $(z, j)(w, i_w)=(z, j\ast_zi_w)$ and $x\phi_j\phi_{i_w}\sim_\cals x\phi_j$ because $\theta(s(x\phi_j))=\Phi_j(z)=w$ and hence $s(x\phi_j)\phi_{i_w}\in \cals$.
Therefore $j\ast_zi_w=j$.
\end{proof}

\begin{lem}\label{lem-inverse}
For every $g\in \calq$, there exists $h\in \calq$ such that the product $gh$ is defined and we have $gh=u_{r(g)}$ and $hg=u_{s(g)}$.
Moreover the map $g\mapsto h$ is Borel.
\end{lem}

\begin{proof}
For all $x\in X$ and $j\in J$, there exists a unique $k=\kappa(x, j)\in J$ such that $x\phi_j\phi_k\in \cals$ (or equivalently $s(x\phi_j)\phi_k\sim_\cals (x\phi_j)^{-1}$).
Applying Lemma \ref{lem-s-inv} (i) to $\phi_j\phi_k\in \End_\calg(\cals)$ for fixed $j$ and $k$ and $\calg^0\in \End_\calg(\cals)$, we see that the map $(x, j)\mapsto \kappa(x, j)$ induces the Borel map $\kappa \colon Z\times J\to J$, denoted by the same symbol $\kappa$.

Fix $x\in X$ and $j\in J$, and put $z=\theta(x)$ and $k=\kappa(x, j)$.
Then
\[(z, j)(\Phi_j(z), k)=(z, j\ast_zk)=(z, i_z),\]
where the last equation holds because $x\phi_j\phi_k\in \cals$ and hence $x\phi_j\phi_k\sim_\cals x\phi_{i_x}$.
The source of $(\Phi_j(z), k)$ is $z$, and the product $(\Phi_j(z), k)(z, j)$ is defined.
Putting $y=s(x\phi_j)$, we have $y\phi_k\phi_j=\phi_j^{-1}\gamma \phi_j$, where $\gamma \coloneqq x\phi_j\phi_k\in \cals$.
Since $\phi_j\in \End_\calg(\cals)$, we have $y\phi_k\phi_j\in \cals$.
Thus $k\ast_{\Phi_j(z)}j=i_{\Phi_j(z)}$ and $(\Phi_j(z), k)(z, j)=(\Phi_j(z), i_{\Phi_j(z)})$.

We proved that the inverse of $(z, j)\in \calq$ is $(\Phi_j(z), k)$.
It follows from measurability of the maps $\Phi_j$ and $\kappa$ that the inverse map $(z, j)\mapsto (z, j)^{-1}$ is Borel.
\end{proof}

We have constructed the discrete Borel groupoid $\calq$.

\begin{lem}
The pair $(\calq, \zeta)$ is a discrete measured groupoid.
\end{lem}

\begin{proof}
In the proof of Lemma \ref{lem-inverse}, we obtained the Borel map $\kappa \colon Z\times J\to J$ such that $(z, j)^{-1}=(\Phi_j(z), \kappa(z, j))$ for all $z\in Z$ and $j\in J$.
It follows that for every $j\in J$, there exists a countable Borel partition $Z=\bigsqcup_{k\in J}Z_{k}$ such that $\Phi_k\circ \Phi_j$ is the identity on $Z_k$.
Therefore the image of every Borel subset of $Z_k$ under $\Phi_j$ is a Borel subset of $Z$, and in particular $\Phi_j(Z_k)$ is Borel.
In the following commutative diagram
\[\xymatrix{
\theta^{-1}(Z_k) \ar[r]^{\bar{\phi}_j} \ar[d]_{\theta} & \theta^{-1}(\Phi_j(Z_k)) \ar[d]^\theta \\
Z_k \ar[r]_{\Phi_j} & \Phi_j(Z_k)}
\]
the isomorphism $\Phi_j\colon Z_k\to \Phi_j(Z_k)$ preserves the measure class of $\zeta$.
Indeed, since $\theta_*\mu =\zeta$, if $A\subset \Phi_j(Z_k)$ is a null set, then the set $\theta^{-1}(\Phi_j^{-1}(A))=\bar{\phi}_j^{-1}(\theta^{-1}(A))$ is null, and $\Phi_j^{-1}(A)$ is null.
If $B\subset Z_k$ is a null set, then the set $\bar{\phi}_j^{-1}(\theta^{-1}(\Phi_j(B)))=\theta^{-1}(\Phi_j^{-1}(\Phi_j(B)))=\theta^{-1}(B)$ is null, and $\Phi_j(B)$ is null.

Let $\phi$ be a Borel bisection of $\calq$.
After partitioning $\phi$ into countably many Borel subsets, we may assume $\bar{\phi}=\Phi_j$ on $r(\phi)$ and $r(\phi)\subset Z_k$ for some $j$ and $k$.
The above argument shows that $\bar{\phi}_*(\zeta|_{r(\phi)})$ and $\zeta|_{s(\phi)}$ are equivalent.
\end{proof}

We define a map $\theta \colon \calg \to \calq$ by $\theta(g)=(\theta(r(g)), j)$ for $g\in \calq$, where $j\in J$ is determined by $g\sim_\cals r(g)\phi_j$.
In particular we have $\theta(x\phi_j)=(\theta(x), j)$.
This implies that $\theta$ satisfies the class-surjectivity condition stated in the theorem.

We show that $\theta$ is a homomorphism.
For each $x\in \calg^0$, write $\bar{x}=\theta(x)$ for simplicity.
Pick $g, h\in \calg$ with $gh$ defined.
We put $x=r(g)$, $y=r(h)$ and define $j, k\in J$ by $\theta(g)=(\bar{x}, j)$ and $\theta(h)=(\bar{y}, k)$.
Since $g\sim_\cals x\phi_j$, the two points $s(g)=y$ and $s(x\phi_j)$ are in the same $\theta$-fiber and hence $\Phi_j(\bar{x})=\bar{y}$.
It follows that the product $(\bar{x}, j)(\bar{y}, k)$ is defined and is equal to $(\bar{x}, j\ast_{\bar{x}}k)$.
We have
\[x\phi_j\phi_k=g\phi_k\delta,\ \  \text{where}\ \  \delta \coloneqq \phi_k^{-1}g^{-1}\phi_j\phi_k.\]
By the definition of $\theta$, we have $g^{-1}\phi_j, h^{-1}\phi_k\in \cals$.
Since $\phi_k\in \End_\calg(\cals)$, we have $\delta \in \cals$ and hence $gh\sim_\cals g\phi_k\delta =x\phi_j\phi_k$.
Thus $\theta(gh)=(\bar{x}, j\ast_{\bar{x}}k)=\theta(g)\theta(h)$.

For every $g\in \calg$ with $x=r(g)$, we have $g\in \ker \theta$ if and only if $\theta(g)=(\bar{x}, i_{\bar{x}})$.
By the definition of $\theta$, this holds if and only if $g\sim_\cals x\phi_{i_{\bar{x}}}$, i.e., $g\in \cals$.
Thus $\ker \theta =\cals$.

Finally we prove the universal property in condition (3) of the theorem.
Let $\calq'$ be a discrete Borel groupoid and let $\theta' \colon \calg \to \calq'$ be a Borel homomorphism such that $\cals <\ker \theta'$.
For every $j\in J$, the map $\calg^0\to \calq'$, $x\mapsto \theta'(x\phi_j)$ is $\cals$-invariant.
Indeed for every $\gamma \in \cals$, we have $\phi_j^{-1}\gamma \phi_j\in \cals$ and $\theta'(r(\gamma)\phi_j)=\theta'(s(\gamma)\phi_j)$ since $\cals <\ker \theta'$.
Therefore we obtain a Borel map $\tau \colon \calq \to \calq'$ such that $\tau(z, j)=\theta'(x\phi_j)$ for all $z\in Z$ and $j\in J$ with $\theta(x)=z$. 
It follows from $(\tau \circ \theta)(x\phi_j)=\tau(z, j)=\theta'(x\phi_j)$ that $\tau \circ \theta =\theta'$.

To verify that $\tau$ is a homomorphism, pick $z\in Z$ and $j, k\in J$, and put $l=j\ast_zk$.
Then for every $x\in \theta^{-1}(z)$, we have $x\phi_l\sim_\cals x\phi_j\phi_k$ and
\begin{align*}
\tau((z, j)(\Phi_j(z), k))&=\tau(z, l)=\theta'(x\phi_l)=\theta'(x\phi_j\phi_k)=\theta'(x\phi_j)\theta'(s(x\phi_j)\phi_k)\\
&=\tau(z, j)\tau(\Phi_j(z), k).\qedhere
\end{align*}
\end{proof}

The converse assertion of Theorem \ref{thm-quotient} holds.
Namely, if $(\calg, \mu)$ is a discrete measured groupoid, $\calq$ is a discrete Borel groupoid, and $\theta \colon \calg \to \calq$ is a Borel, class-surjective homomorphism, then $\ker \theta \vartriangleleft \calg$.
The proof of this assertion is postponed to Proposition \ref{prop-bowen}, where a stronger assertion is proved.


\subsection{A criterion for quotient groupoids being p.m.p.}

Let $(\calg, \mu)$ be a discrete measured groupoid and let $\cals <\calg$ be a Borel subgroupoid.
We define $\llbracket \calg \rrbracket_\cals$ as the set of all $\phi \in \llbracket \calg \rrbracket$ such that the associated conjutation $V_\phi \colon \calg|_{r(\phi)}\to \calg|_{s(\phi)}$, $\gamma \mapsto \phi^{-1}\gamma \phi$ satisfies the equation $V_\phi(\cals|_{r(\phi)})=\cals|_{s(\phi)}$.
If $\phi, \psi \in \llbracket \calg \rrbracket_\cals$, then $\phi \psi \in \llbracket \calg \rrbracket_\cals$.
The following is inspired by \cite[Theorem 2.15]{fsz}.

\begin{prop}\label{prop-pmp}
Let $(\calg, \mu)$ be a discrete p.m.p.\ groupoid and let $\cals \vartriangleleft \calg$ be a normal Borel subgroupoid.
Let $(\calq, \zeta)$ be the quotient groupoid of $(\calg, \mu)$ by $\cals$.
Then the following two conditions are equivalent:
\begin{enumerate}
\item[(1)] The groupoid $(\calq, \zeta )$ is p.m.p.
\item[(2)] There exists a sequence $(\phi_n)_{n\in \N}$ of $\phi_n \in \llbracket \calg \rrbracket_\cals$ such that $r(\phi_n)$ and $s(\phi_n)$ are both $\cals$-invariant, and for almost every $x\in \calg^0$, the set $\{ \, x\phi_n\mid x\in r(\phi_n),\, n\in \N \, \}$ intersects every $\sim_\cals$-equivalence class in $x\calg$.
\end{enumerate}
\end{prop}

\begin{rem}
We define $\aut_\calg(\cals)=[\calg]\cap \llbracket \calg \rrbracket_\cals$.
The proof of the proposition given below shows that if the condition in the proposition holds, then one can find a sequence $(\phi_n)_{n\in \N}$ of $\phi_n \in \aut_\calg(\cals)$ satisfying condition (2).
\end{rem}

\begin{proof}[Proof of Proposition \ref{prop-pmp}]
We may assume that the index function for $\cals <\calg$ is constant.
We put $Z=\calq^0$ and let $\theta \colon \calg \to \calq$ denote the quotient homomorphism with $\theta_*\mu =\zeta$.
We first suppose that there exists a sequence $(\phi_n)_{n\in \N}$ in condition (2).
For every $n\in \N$, the map $r(\phi_n)\to \calq$, $x\mapsto \theta(x\phi_n)$ is $\cals$-invariant.
Indeed for every $\gamma \in \cals$ with $s(\gamma), r(\gamma)\in r(\phi_n)$, we have $\phi_n^{-1}\gamma \phi_n\in \cals$ and hence $\theta(s(\gamma)\phi_n)=\theta(r(\gamma)\phi_n)$.
Since $r(\phi_n)$ is $\cals$-invariant, we have a Borel subset $Z_n\subset Z$ with $\theta^{-1}(Z_n)=r(\phi_n)$.
We obtain a Borel map $\psi_n\colon Z_n\to \calq$ such that $\psi_n(\theta(x))=\theta(x\phi_n)$ for all $x\in r(\phi_n)$, which is a section of the range map of $\calq$.
We identify this $\psi_n$ with its image.

We claim that $\psi_n\in \llbracket \calq \rrbracket$.
Since $s(\phi_n)$ is $\cals$-invariant, there exists a Borel subset $W_n\subset Z$ such that $\theta^{-1}(W_n)=s(\phi_n)$.
As $\phi_n$ induces $\psi_n$, $\phi_n^{-1}$ induces the $r$-section $\xi_n$ of $\calq$ such that $r(\xi_n)=W_n$ and $\theta(y)\xi_n=\theta(y\phi_n^{-1})$ for all $y\in s(\phi_n)$.
Then for almost every $x\in \theta^{-1}(Z_n)$, putting $y=s(x\phi_n)$, we have $y\phi_n^{-1}=(x\phi_n)^{-1}$ and
\begin{align*}
\bar{\xi}_n(\bar{\psi}_n(\theta(x)))&=\bar{\xi}_n(s(\theta(x\phi_n)))=\bar{\xi}_n(\theta(y))=s(\theta(y)\xi_n)=s(\theta(y\phi_n^{-1}))\\
&=r(\theta(x\phi_n))=\theta(x).
\end{align*}
It therefore follows that $\bar{\xi}_n\circ \bar{\psi}_n$ is the identity almost everywhere on $Z_n$.
In particular $\bar{\psi}_n$ is injective after discarding a null set and hence $\psi_n\in \llbracket \calq \rrbracket$.
By symmetry $\bar{\psi}_n\circ \bar{\xi}_n$ is the identity almost everywhere on $W_n$.

In the following commutative diagram, the two downward maps are measure-preserving since $r(\phi_n)$ and $s(\phi_n)$ are $\cals$-invariant:
\[\xymatrix{
r(\phi_n) \ar[r]^{\bar{\phi}_n} \ar[d]_{\theta} & s(\phi_n) \ar[d]^\theta \\
Z_n \ar[r]_{\bar{\psi}_n} & W_n}
\]
Since $(\calg, \mu)$ is p.m.p.\ and $\bar{\phi}_n$ is measure-preserving, the isomorphism $\bar{\psi}_n\colon Z_n\to W_n$ is measure-preserving.
By the assumption that the set $\{ \, x\phi_n\mid x\in r(\phi_n),\, n\in \N \, \}$ intersects every $\sim_\cals$-equivalence class in $x\calg$, we have $\calq =\bigcup_n\psi_n$.
Thus $(\calq, \zeta)$ is p.m.p.

Conversely we suppose that $(\calq, \zeta)$ is p.m.p.
Choose a sequence $(\psi_n)_{n\in \N}$ of $\psi_n\in [ \calq ]$ such that $\calq =\bigcup_n \psi_n$.
We may assume that for every $n$, $\psi_n^{-1}=\psi_m$ for some $m$.

\begin{claim}\label{claim-phi_n}
For every $n$, there exists $\phi_n\in [\calg ]$ such that $\theta(x\phi_n)=\theta(x)\psi_n$ for almost every $x\in \calg^0$.
\end{claim}

\begin{proof}
Fix $n$ and choose a Borel $r$-section $\phi_n\subset \theta^{-1}(\psi_n)$ of $\calg$ with $r(\phi_n)=\calg^0$.
Take a countable Borel partition $\calg^0=\bigsqcup_kE_k$ such that $\bar{\phi}_n$ is injective on each $E_k$.
Let $\mu =\int_Z\mu_z\, d\zeta(z)$ be the disintegration with respect to the ergodic decomposition map $\theta$.
The maps $\bar{\phi}_n\colon E_k\to \bar{\phi}_n(E_k)$ and $\bar{\psi}_n\colon Z\to Z$ are measure-preserving isomorphisms. 
For every $x\in \calg^0$, we have $\theta(x\phi_n)\in \psi_n$ with the range $\theta(x)$ and hence $\theta(x\phi_n)=\theta(x)\psi_n$.
It follows that $\theta \circ \bar{\phi}_n=\bar{\psi}_n\circ \theta \colon \calg^0\to Z$.

For every $k$ and almost every $z\in Z$, we have
\begin{equation}\label{eq-bar-psi-bar-phi}
\mu_{\bar{\psi}_n(z)}(\bar{\phi}_n(E_k))=\mu_z(E_k)
\end{equation}
since for every Borel subset $A\subset Z$, we have
\begin{align*}
&\int_A\mu_{\bar{\psi}_n(z)}(\bar{\phi}_n(E_k))\, d\zeta(z)=\int_{\bar{\psi}_n(A)}\mu_w(\bar{\phi}_n(E_k))\, d\zeta (w)=\mu(\theta^{-1}(\bar{\psi}_n(A))\cap \bar{\phi}_n(E_k))\\
&=\mu(\bar{\phi}_n(\theta^{-1}(A)\cap E_k))=\mu(\theta^{-1}(A)\cap E_k)=\int_A\mu_z(E_k)\, d\zeta (z).
\end{align*}
The first equation holds because $(\calq, \zeta)$ is p.m.p.
For the third equation, we in fact have the equation
\[\theta^{-1}(\bar{\psi}_n(A))\cap \bar{\phi}_n(E_k)=\bar{\phi}_n(\theta^{-1}(A)\cap E_k)\]
as verified as follows:
If $x\in \theta^{-1}(A)\cap E_k$, then $\theta(\bar{\phi}_n(x))=\bar{\psi}_n(\theta(x))\in \bar{\psi}_n(A)$ and hence $\bar{\phi}_n(x)\in \theta^{-1}(\bar{\psi}_n(A))$.
Conversely if $y\in \theta^{-1}(\bar{\psi}_n(A))\cap \bar{\phi}_n(E_k)$, then we have $x\in E_k$ such that $y=\bar{\phi}_n(x)$ and hence $\theta(y)=\theta(\bar{\phi}_n(x))=\bar{\psi}_n(\theta(x))$.
Since $\theta(y)\in \bar{\psi}_n(A)$ and the map $\bar{\psi}_n\colon Z\to Z$ is an isomorphism, we have $\theta(x)\in A$ and $y=\bar{\phi}_n(x)$ with $x\in \theta^{-1}(A)\cap E_k$.

It follows from equation (\ref{eq-bar-psi-bar-phi}) that for almost every $z\in Z$, 
\begin{equation}\label{eq-sum-k}
\sum_k\mu_{\bar{\psi}_n(z)}(\bar{\phi}_n(E_k))=1.
\end{equation}
By ergodicity of $(\cals, \mu_w)$, we have $\mu_{\bar{\psi}_n(z)}([\bar{\phi}_n(E_k)]_\cals)=1$ if $\mu_{\bar{\psi}_n(z)}(\bar{\phi}_n(E_k))>0$.
Using this, we find a countable Borel partition $\calg^0=\bigsqcup_l E_l'$ finer than $\calg^0=\bigsqcup_k E_k$ and find $\eta_l'\in \llbracket \cals \rrbracket$ such that the $r$-section $\xi_n\coloneqq \bigsqcup_l \phi_n\eta_l'$ belongs to $[ \calg ]$.
This $\xi_n$ is a desired replacement of $\phi_n$ in the claim.

We will define $x\xi_n\in x\calg$ for each $x\in \calg^0$. 
Choose a sequence $(\eta_m)_{m\in \N}$ of $\eta_m\in [\cals]$ such that $\cals =\bigcup_{m\in \N}\eta_m$.
We set $x\xi_n=x\phi_n$ for $x\in E_1$.
Define the set $F_2\subset E_2$ by the equation $\bar{\phi}_n(F_2)=\bar{\phi}_n(E_1)\cap \bar{\phi}_n(E_2)$.
We will find a decomposition $F_2=\bigsqcup_m F_{2, m}$ such that $(\bar{\eta}_m(\bar{\phi}_n(F_{2, m})))_m$ is mutually disjoint and every $\bar{\eta}_m(\bar{\phi}_n(F_{2, m}))$ is disjoint from the union $\bar{\phi}_n(E_1)\cup \bar{\phi}_n(E_2)$.

We may assume that $F_2$ is nonnull.
Put $A_1=\bar{\phi}_n(E_1)$ and $A_2=\bar{\phi}_n(E_2)$.
By equation (\ref{eq-sum-k}), we have
\[\mu_{\bar{\psi}_n(z)}(\calg^0\setminus (A_1\cup A_2))\geq \mu_{\bar{\psi}_n(z)}(A_1\cap A_2)\]
for almost every $z\in Z$.
Since $\mu_{\bar{\psi}_n(z)}([A_1\cap A_2]_\cals)=1$ if $\mu_{\bar{\psi}_n(z)}(A_1\cap A_2)>0$, we obtain a decomposition $A_1\cap A_2=\bigsqcup_m D_m$ such that $(\bar{\eta}_m(D_m))_m$ is mutually disjoint and every $\bar{\eta}_m(D_m)$ is disjoint from $A_1\cup A_2$.
More precisely we obtain the sets $D_m$ as follows:
Define $D_1$ by $\bar{\eta}_1(D_1)=\bar{\eta}_1(A_1\cap A_2)\setminus (A_1\cup A_2)$.
For $m\geq 2$, define $D_m$ inductively by
\[\bar{\eta}_m(D_m)=\bar{\eta}_m(A_1\cap A_2)\setminus (A_1\cup A_2\cup \bar{\eta}_1(D_1)\cup \cdots \cup \bar{\eta}_{m-1}(D_{m-1})).\]
Via the isomorphism $\bar{\phi}_n\colon F_2\to A_1\cap A_2$, we obtain the decomposition $F_2=\bigsqcup_m F_{2, m}$ such that $\bar{\phi}_n(F_{2, m})=D_m$.
This is a desired one.

We set $x\xi_n =x\phi_n \eta_m$ for $x\in F_{2, m}$ and $m\in \N$, and set $x\xi_n=x\phi_n$ for $x\in E_2\setminus F_2$.
We defined $x\xi_n$ for $x\in E_1\cup E_2$.
The map $\bar{\xi}_n\colon E_1\cup E_2\to \calg^0$, $x\mapsto s(x\xi_n)$ is then injective.

We next define $x\xi_n$ for $x\in E_3$ in a similar manner as follows.
Define the set $F_3\subset E_3$ by the equation $\bar{\phi}_n(F_3)=\bar{\xi}_n(E_1\cup E_2)\cap \bar{\phi}_n(E_3)$, and put $A_{12}=\bar{\xi}_n(E_1\cup E_2)$ and $A_3=\bar{\phi}_n(E_3)$.
Apply the same procedure as above to $A_{12}$ and $A_3$ in place of $A_1$ and $A_2$.
We then obtain a decomposition $A_{12}\cap A_3=\bigsqcup_m D_m'$ such that $(\bar{\eta}_m(D_m'))_m$ is mutually disjoint and every $\bar{\eta}_m(D_m')$ is disjoint from $A_{12}\cup A_3$.
Let $F_3=\bigsqcup_m F_{3, m}$ be the decomposition obtained by sending this decomposition via the isomorphism $\bar{\phi}_n \colon F_3\to A_{12}\cap A_3$.
We set $x\xi_n=x\phi_n\eta_m$ for $x\in F_{3, m}$ and $m\in \N$, and set $x\xi_n =x\phi_n$ for $x\in E_3\setminus F_3$.
Then $x\xi_n$ is defined for all $x\in E_1\cup E_2\cup E_3$, and the induced map $\bar{\xi}_n\colon E_1\cup E_2\cup E_3\to \calg^0$ is injective.

Repeating this construction, we obtain a decomposition $\calg^0=\bigsqcup_kE_k =\bigsqcup_m F_m'\sqcup F'$ into Borel subsets such that if we set $x\xi_n =x\phi_n \eta_m$ for $x\in F_m'$ and set $x\xi_n=x\phi_n$ for $x\in F'$, then the induced map $\bar{\xi}_n\colon \calg^0\to \calg^0$ is injective.
Thus $\xi_n\in [\calg]$. 
\end{proof}

Going back to the proof of Proposition \ref{prop-pmp} (1) $\Rightarrow$ (2), we show that the sequence $(\phi_n)$ in Claim \ref{claim-phi_n} is a desired one.
Since $\calq =\bigcup_n\psi_n$, the set $\{ \, x\phi_n\mid x\in r(\phi_n),\, n\in \N \, \}$ intersects every $\sim_\cals$-equivalence class in $x\calg$ for almost every $x\in \calg^0$.
For $\gamma \in \cals$, we have $\theta(\phi_n^{-1}\gamma \phi_n)=\psi_n^{-1}\theta(\gamma)\psi_n$, which is a unit.
Therefore $V_{\phi_n}(\cals )\subset \cals$ for every $n$.

We show the equation $V_{\phi_n}(\cals )= \cals$ for every $n$.
Fix $n$.
By the condition we first imposed on $(\psi_n)_n$, there exists $m$ such that $\psi_n^{-1}=\psi_m$.
Then $V_{\phi_m}V_{\phi_n}(\cals)\subset V_{\phi_m}(\cals)\subset \cals$.
Since $\phi_n\phi_m\subset \cals$, we have $V_{\phi_m}V_{\phi_n}(\cals)=\cals$ and hence $V_{\phi_n}(\cals)=\cals$.
Thus $\phi_n \in \aut_\calg(\cals)$.
\end{proof}


\subsection{Useful criteria for normality}

Let $(\calg, \mu)$ be a discrete measured groupoid and let $\cals <\calg$ be a Borel subgroupoid.
For a Borel subset $D\subset \calg^0$, we define $\End_\calg(\cals)|_D$ as the set of Borel $r$-sections $\phi$ of $\calg$ such that $r(\phi)=D$ and $\phi^{-1}\gamma \phi \in \cals$ for every $\gamma \in \cals|_{D}$.
If $\phi \in \llbracket \calg \rrbracket_\cals$, then $\phi \in \End_\calg(\cals)|_{r(\phi)}$.

\begin{lem}\label{lem-phi-ext}
Let $(\calg, \mu)$ be a discrete measured groupoid and let $\cals <\calg$ be a Borel subgroupoid.
Let $D\subset \calg^0$ be a Borel subset.
Then every element of $\End_\calg(\cals)|_D$ extends to an element of $\End_\calg(\cals)$. 
\end{lem}

\begin{proof}
Pick $\phi \in \End_\calg(\cals)|_D$ and put $E=[D]_\cals$.
Choose a Borel $r$-section $f\subset \cals$ such that $r(f)=E$, $xf=x$ for every $x\in D$, and $s(xf)\in D$ for every $x\in E$. 
We define an $r$-section $\tilde{\phi}$ of $\calg$ by $\tilde{\phi}=\{ \, xf\phi \mid x\in E\, \} \cup (\calg^0\setminus E)$.
Then $\tilde{\phi}$ extends $\phi$.
We check that $\tilde{\phi}$ belongs to $\End_\calg(\cals)$.
Pick $\gamma \in \cals$.
Since $E$ is $\cals$-invariant, we have either $\gamma \in \cals|_E$ or $\gamma \in \cals|_{\calg^0\setminus E}$.
If $\gamma \in \cals|_E$, then $\tilde{\phi}^{-1}\gamma \tilde{\phi}=\phi^{-1}f^{-1}\gamma f\phi$, which belongs to $\cals$ because $f^{-1}\gamma f\in \cals$.
If $\gamma \in \cals|_{\calg^0\setminus E}$, then $\tilde{\phi}^{-1}\gamma \tilde{\phi}=\gamma \in \cals$.
\end{proof}

\begin{lem}\label{lem-cri-nor}
Let $(\calg, \mu)$ be a discrete measured groupoid and let $\cals <\calg$ be a Borel subgroupoid.
Assume that there exists a sequence $(\phi_n)$ such that for every $n$, $\phi_n\in \End_\calg(\cals)|_{D_n}$ for some Borel subset $D_n\subset \calg^0$ and $\calg =\bigcup_n \phi_n$.
Then $\cals \vartriangleleft \calg$.
\end{lem}

\begin{proof}
We may assume that the index function for $\cals <\calg$ is constant, and the constant is denoted by $N\in \N \cup \{ \infty \}$.
By Lemma \ref{lem-phi-ext}, we extend each $\phi_n$ to an element of $\End_\calg(\cals)$, which is denoted by the same symbol $\phi_n$.
We define an $r$-section $\psi_n$ of $\calg$ with $r(\psi_n)=\calg^0$ inductively as follows:
First we set $\psi_1=\phi_1$.
Suppose that $\psi_1,\ldots, \psi_{n-1}$ are defined.
For $x\in \calg^0$, we set $x\psi_n=x\phi_k$, where
\[k\coloneqq \min \{ \, l\in \N \mid x\phi_l\not\in [x\psi_1]\cup \cdots \cup [x\psi_{n-1}]\, \},\]
and for $\gamma \in x\calg$, $[\gamma ]$ denotes the $\sim_\cals$-equivalence class in $x\calg$ containing $\gamma$.
Then the family $(\psi_n)_{n=1}^N$ is defined and is a family of choice functions for $\cals <\calg$.

We show $\psi_n\in \End_\calg(\cals)$ for every $n$ by induction.
We have $\psi_1=\phi_1\in \End_\calg(\cals)$.
Next assume $\psi_1,\ldots, \psi_{n-1}\in \End_\calg(\cals)$.
For every $l\in \N$, the set
\[\{ \, x\in \calg^0\mid x\phi_l\not\in [x\psi_1]\cup \cdots \cup [x\psi_{n-1}]\, \}\]
is $\cals$-invariant.
Indeed by Lemma \ref{lem-s-inv} (i), the set $\{ \, x\in \calg^0\mid x\phi_l\not\in [x\psi_i]\, \}$ is $\cals$-invariant for every $i=1,\ldots, n-1$.
Therefore there exists a decomposition $\calg^0=\bigsqcup_{l\in \N}D_l$ such that every $D_l$ is an $\cals$-invariant Borel subset and $x\psi_n =x\phi_l$ for all $x\in D_l$.
Since $\phi_l\in \End_\calg(\cals)$ and $D_l$ is $\cals$-invariant, we conclude that $\psi_n\in \End_\calg(\cals)$.
\end{proof}

\begin{cor}\label{cor-cri-nor-gen}
Let $(\calg, \mu)$ be a discrete measured groupoid and let $\cals <\calg$ be a Borel subgroupoid.
Assume that there exists a sequence $(\phi_n)$ of $\phi_n\in \llbracket \calg \rrbracket_\cals$ such that $(\phi_n)$ and $\cals$ generate $\calg$. 
Then $\cals \vartriangleleft \calg$.
\end{cor}

\begin{proof}
We may assume that the sequence $(\phi_n)$ contains the inverse of every $\phi_n$.
Then $\calg$ is written as the union of countably many words over $(\phi_n)$ and elements of $\llbracket \cals \rrbracket$.
Those words belong to $\llbracket \calg \rrbracket_\cals$, and the lemma follows from Lemma \ref{lem-cri-nor}.
\end{proof}

\begin{cor}\label{cor-shg}
Let $(\calg, \mu)$ be a discrete measured groupoid.
If $\cals <\calh <\calg$ are Borel subgroupoids of $\calg$ such that $\cals \vartriangleleft \calg$, then $\cals \vartriangleleft \calh$.
\end{cor}

\begin{proof}
Let $(\phi_n)$ be a family of choice functions for $\cals \vartriangleleft \calg$ such that $\phi_n\in \End_\calg(\cals)$ for all $n$.
Write $\cals =\bigcup_m \eta_m$ as the union of countably many $\eta_m\in [\cals]$.
Then $\calg =\bigcup_{n, m}\phi_n \eta_m$.
For all $n$ and $m$, $(\phi_n\eta_m)\cap \calh$ belongs to $\End_\calh(\cals)|_{D_{n, m}}$ for some Borel subset $D_{n, m}\subset \calg^0$ and $\calh =\bigcup_{n, m} (\phi_n\eta_m) \cap \calh$.
By Lemma \ref{lem-cri-nor}, we have $\cals \vartriangleleft \calh$.
\end{proof}

\begin{lem}\label{lem-conv-cri-nor}
The converse assertion of Lemma \ref{lem-cri-nor} also holds.
In fact, if $(\calg, \mu)$ is a discrete measured groupoid and $\cals \vartriangleleft \calg$ is a normal Borel subgroupoid, then there exists a sequence $(\phi_n)$ of $\phi_n\in \llbracket \calg \rrbracket_\cals$ such that $\calg =\bigcup_n \phi_n$.
\end{lem}

\begin{proof}
We may assume that the index function for $\cals <\calg$ is constant, and let $J$ be a set whose cardinality is the index for $\cals <\calg$.
Since $\cals \vartriangleleft \calg$, there exists a family $(\phi_j)_{j\in J}$ of choice functions for $\cals <\calg$ such that $\phi_j\in \End_\calg(\cals)$ for every $j\in J$.
As in the proof of Lemma \ref{lem-inverse}, there exists a Borel map $\kappa \colon Z\times J\to J$ such that $x\phi_j\phi_{\kappa(\theta(x), j)}\in \cals$ for almost every $x\in \calg^0$, where $\theta \colon (\calg^0, \mu)\to (Z, \zeta)$ is the ergodic decomposition map for $(\cals, \mu)$.

Fix $j\in J$.
There exists a countable Borel partition $Z=\bigsqcup_{k\in J}Z_k$ such that $x\phi_j\phi_k\in \cals$ for almost every $x\in \theta^{-1}(Z_k)$.
For every $k\in J$, there exists a countable Borel partition $\theta^{-1}(Z_k)=\bigsqcup_nD_{k, n}$ such that $\bar{\phi}_j$ is injective on $D_{k, n}$ and $\bar{\phi}_k$ is injective on $\bar{\phi}_j(D_{k, n})$.
Then we have $V_{\phi_j}(\cals|_{D_{k, n}})\subset \cals|_{\bar{\phi}_j(D_{k, n})}$ and
\[\cals|_{\bar{\phi}_k( \bar{\phi}_j(D_{k, n}))}=V_{\phi_k}V_{\phi_j}(\cals|_{D_{k, n}})\subset V_{\phi_k}(\cals|_{\bar{\phi}_j(D_{k, n})})\subset \cals|_{\bar{\phi}_k( \bar{\phi}_j(D_{k, n}))}.\]
Therefore $V_{\phi_j}(\cals|_{D_{k, n}})=\cals|_{\bar{\phi}_j(D_{k, n})}$ for all $k\in J$ and all $n$.
Thus $D_{k, n}\phi_j\in \llbracket \calg \rrbracket_\cals$.
It follows that $\phi_j$ is the union $\phi_j=\bigsqcup_l \phi_{j, l}$ of countably many $\phi_{j, l}\in \llbracket \calg \rrbracket_\cals$.

Choose a sequence $(\eta_m)$ of $\eta_m\in [\cals]$ such that $\cals =\bigcup_m \eta_m$.
Then the family of $\phi_{j, l}\eta_m$ with all $j\in J$, $l$ and $m$ is a desired one.
\end{proof}

The following fact is inspired by \cite[Theorem 2.1]{bowen} and implies the converse assertion of Theorem \ref{thm-quotient}.

\begin{prop}\label{prop-bowen}
Let $(\calg, \mu)$ be a discrete measured groupoid and let $\calq$ be a discrete Borel groupoid.
Let $\theta \colon \calg \to \calq$ be a Borel homomorphism.
Then $\ker \theta \vartriangleleft \calg$.
\end{prop}

\begin{proof}
We put $\cals =\ker \theta$. 
The function on $\calq^0$, $z\mapsto |z\calq |$, is $\calq$-invariant.
We may assume that the value $|z\calq |$ is constant on $\calq^0$ after restricting $\calq$ to a $\calq$-invariant subset of $\calq^0$ on which this function is constant and replacing $\calg$ with its inverse image under $\theta$.
Let $J$ be a countable set whose cardinality is that constant.
By the Luzin-Novikov uniformization theorem \cite[Theorem 18.10]{kechris}, there exists a family $(\psi_j)_{j\in J}$ of Borel $r$-sections $\psi_j$ of $\calq$ such that $r(\psi_j)=\calq^0$ and $\calq =\bigsqcup_{j\in J}\psi_j$.

For every $x\in \calg^0$, we have the map $\theta \colon x\calg \to \theta(x)\calq =\{ \, \theta(x)\psi_j\mid j\in J\, \}$.
For $j\in J$, we set
\[X_j= \{ \, x\in \calg^0 \mid \theta(x)\psi_j\in \theta(x\calg)\, \},\]
Since $X_j=r(\theta^{-1}(\psi_j))$, it is Borel.
Pick a Borel $r$-section $\phi_j\subset \theta^{-1}(\psi_j)$ with $r(\phi_j)=X_j$.
Then $\phi_j\in \End_\calg(\cals)|_{X_j}$.
Indeed if $\gamma \in \cals|_{X_j}$, then $\theta(\phi_j^{-1}\gamma \phi_j)=\psi_j^{-1}\theta(\gamma)\psi_j$, which is a unit, and hence $\phi_j^{-1}\gamma \phi_j\in \ker \theta =\cals$.

Write $\calg$ as the union of the products $\phi_j\eta$ for all $j\in J$ and countably many $\eta \in [\cals]$.
By Lemma \ref{lem-cri-nor}, we conclude $\cals \vartriangleleft \calg$.
\end{proof}


\subsection{An example}

We give a nontrivial example of normal subgroupoids which does not come from normal subgroups.
The following lemma generalizes \cite[Proposition B.1]{kida-bs}.

\begin{lem}\label{lem-ex-nor}
Let $G$ be a countable group.
Let $E<G$ be a subgroup and $(t_n)$ a countable family of elements of $G$ such that
\begin{itemize}
\item $E$ and $(t_n)$ generate $G$, and
\item for every $n$, if we put $E_n^-=E\cap t_n^{-1}Et_n$ and $E_n^+=t_nEt_n^{-1}\cap E$, then both $E_n^-$ and $E_n^+$ are finite-index normal subgroups of $E$.
\end{itemize}
Let $G\c X$ be a Borel action on a standard Borel space $X$ which preserves the class of a probability measure $\mu$ on $X$.
We assume that for every $n$, there exist $E$-equivariant Borel maps $X\to E/E_n^-$ and $X\to E/E_n^+$.
Then $X\rtimes E$ is normal in $(X\rtimes G, \mu)$. 
\end{lem}

\begin{proof}
We put $\calg =X\rtimes G$ and $\cale =X\rtimes E$.
For each $n$, define $\phi_n\in [\calg]$ by $\phi_n=X\times \{ t_n\}$.
By the existence of the $E$-equivariant maps in the lemma, we have a decomposition
\[X=X_1^-\sqcup \cdots \sqcup X_K^-=X_1^+\sqcup \cdots \sqcup X_L^+\]
into Borel sets, where $K\coloneqq |E/E_n^-|$ and $L\coloneqq |E/E_n^+|$, such that every $X_k^-$ is $E_n^-$-invariant and every $X_l^+$ is $E_n^+$-invariant.

We claim that $\phi_n\cap \calg|_{t_n(X_k^-)\cap X_l^+}$ belongs to $\llbracket \calg \rrbracket_\cale$ for all $k$ and $l$.
By Corollary \ref{cor-cri-nor-gen}, this is enough to imply the lemma.
Since $\cale|_{X_l^+}=X\rtimes E_n^+|_{X_l^+}$, we have
\[\cale|_{t_n(X_k^-)\cap X_l^+}=X\rtimes E_n^+|_{t_n(X_k^-)\cap X_l^+}.\]
Hence
\begin{align*}
V_{\phi_n}(\cale|_{t_n(X_k^-)\cap X_l^+})&=V_{\phi_n}(X\rtimes E_n^+|_{t_n(X_k^-)\cap X_l^+})=X\rtimes E_n^-|_{X_k^-\cap t_n^{-1}(X_l^+)}\\
&=\cale|_{X_k^-\cap t_n^{-1}(X_l^+)},
\end{align*}
where the second equation follows from $t_n^{-1}E_n^+t_n=E_n^-$, and the third equation follows from $X\rtimes E_n^-|_{X_k^-}=\cale|_{X_k^-}$.
\end{proof}

We give an example which will be focused later.

\begin{ex}
Let $E$ be a countable group and let $\tau \colon E_-\to E_+$ be an isomorphism between finite-index normal subgroups of $E$.
We define $G$ as the HNN extension
\[G=\langle \, E,\, t\mid \forall a\in E_-\ \ tat^{-1}=\tau(a)\, \rangle.\]
Let $G\c X$ be a Borel action on a standard Borel space $X$ which preserves the class of a probability measure $\mu$ on $X$.
If there exist $E$-equivariant Borel maps $X\to E/E_-$ and $X\to E/E_+$, then $X\rtimes E$ is normal in $(X\rtimes G, \mu)$ by Lemma \ref{lem-ex-nor}.
\end{ex}

In the rest of this subsection, we keep the notation in Lemma \ref{lem-ex-nor} and suppose that the action $G\c X$ preserves $\mu$.
Let $\theta \colon X\rtimes G\to \calq$ be the quotient homomorphism.
We collect notable properties of $\theta$.
Throughout our discussion in the rest of this subsection, we fix $n$ and hence put $t=t_n$ and $E_\pm =E_n^\pm$ for simplicity.
Let $\theta_\pm \colon (X, \mu)\to (Z_\pm, \zeta_\pm)$ be the ergodic decomposition maps of the p.m.p.\ actions $E_\pm \c (X, \mu)$, respectively.

\begin{lem}\label{lem-eq-erg-comp0}
For almost all $z\in \calq^0$, for almost all $x, x'\in \theta^{-1}(z)$, the following assertions hold:
\begin{enumerate}
\item[(i)] We have $\theta(x, t)=\theta(x', t)$ if and only if $\theta_+(x)=\theta_+(x')$. 
\item[(ii)] We have $\theta(x, t^{-1})=\theta(x', t^{-1})$ if and only if $\theta_-(x)=\theta_-(x')$. 
\end{enumerate}
\end{lem}

\begin{proof}
The map $x\mapsto \theta(x, t)$ is $E_+$-invariant.
Indeed for every $b\in E_+$, we have
\begin{align*}
\theta(bx, t)&=\theta(x, b^{-1})\theta(bx, t)=\theta(x, b^{-1}t)=\theta(x, t\tau^{-1}(b^{-1}))=\theta(x, t)\theta(t^{-1}x, \tau^{-1}(b^{-1}))\\
&=\theta(x, t).
\end{align*}
This proves the ``if'' part of assertion (i). 
For the ``only if'' part, it suffices to show that for each $a\in E\setminus E_+$, we have $\theta(x, t)\neq \theta(ax, t)$.
Since $\theta(ax, t)=\theta(x, a^{-1})\theta(ax, t)=\theta(x, a^{-1}t)$, if $\theta(x, t)= \theta(ax, t)$ held, then the element
\[\theta(x, a^{-1}t)^{-1}\theta(x, t)=\theta(t^{-1}ax, t^{-1}a)\theta(x, t)=\theta(t^{-1}ax, t^{-1}at)\]
would be a unit.
This contradicts $t^{-1}at\not\in E$.
Assertion (ii) follows similarly.
\end{proof}

We have the measure-preserving maps $\sigma_\pm \colon Z_\pm \to Z$ such that $\theta =\sigma_\pm \circ \theta_\pm$.
Almost every fiber of $\sigma_\pm$ has cardinality $[E: E_\pm]$, respectively.
Since $tE_-t^{-1}=E_+$, the action of $t$ on $X$ induces the measure-preserving isomorphism $t\colon Z_-\to Z_+$, which we denote by the same symbol $t$.
Then $t\circ \theta_-=\theta_+\circ t$.

By Lemma \ref{lem-eq-erg-comp0} (i), the map $x\mapsto \theta(x, t)$ is $E_+$-invariant and hence induces the map $Z_+\to \calq$.
We denote this map by $y\mapsto \bar{\theta}(y, t)$, which satisfies the equation $\bar{\theta}(\theta_+(x), t)=\theta(x, t)$ for almost every $x\in X$.
Similarly by Lemma \ref{lem-eq-erg-comp0} (ii), the map $x\mapsto \theta(x, t^{-1})$ is $E_-$-invariant and hence induces the map $Z_-\to \calq$.
We denote this map by $y\mapsto \bar{\theta}(y, t^{-1})$.

\begin{lem}\label{lem-bar-theta}
For almost every $y\in Z_+$, the following assertions hold:
\begin{enumerate}
\item[(i)] The range and source of $\bar{\theta}(y, t)$ are $\sigma_+(y)$ and $\sigma_-(t^{-1}y)$, respectively.
\item[(ii)] The range and source of $\bar{\theta}(t^{-1}y, t^{-1})$ are $\sigma_-(t^{-1}y)$ and $\sigma_+(y)$, respectively.
\item[(iii)] We have $\bar{\theta}(y, t)^{-1}=\bar{\theta}(t^{-1}y, t^{-1})$.
\end{enumerate}
\end{lem}

\begin{proof}
For almost every $x\in X$, we have
\begin{align*}
r(\bar{\theta}(\theta_+(x), t))&=r(\theta(x, t))=\theta(x)=\sigma_+(\theta_+(x))\ \ \text{and}\\
s(\bar{\theta}(\theta_+(x), t))&=s(\theta(x, t))=\theta(t^{-1}x)=\sigma_-(\theta_-(t^{-1}x))=\sigma_-(t^{-1}\theta_+(x)).
\end{align*}
Assertion (i) hence follows.
Assertion (ii) follows similarly.
Assertion (iii) follows from
\[\bar{\theta}(\theta_+(x), t)^{-1}=\theta(x, t)^{-1}=\theta(t^{-1}x, t^{-1})=\bar{\theta}(\theta_-(t^{-1}x), t^{-1})=\bar{\theta}(t^{-1}\theta_+(x), t^{-1})\]
for almost every $x\in X$.
\end{proof}

\begin{lem}\label{lem-inj}
The map $Z_+\to \calq$, $y\mapsto \bar{\theta}(y, t)$ is injective after discarding a null set.
The map $Z_-\to \calq$, $y\mapsto \bar{\theta}(y, t^{-1})$ is also injective after discarding a null set.
\end{lem}

\begin{proof}
The map $Z_+\to \calq$, $y\mapsto r(\bar{\theta}(y, t))=\sigma_+(y)$ is finite-to-one.
The map $y\mapsto \bar{\theta}(y, t)$ is hence finite-to-one.
If the first assertion of the lemma were not true, then there would exist a nonnull Borel subset $A\subset Z_+$ and an injective measure-preserving map $f\colon A\to Z_+$ such that $A\cap f(A)=\emptyset$ and $\bar{\theta}(y, t)=\bar{\theta}(f(y), t)$ for all $y\in A$.
For almost every $x\in \theta_+^{-1}(A)$,
\[\theta(x, t)=\bar{\theta}(\theta_+(x), t)=\bar{\theta}(f(\theta_+(x)), t)=\theta(x', t)\]
for almost every $x'\in \theta_+^{-1}(f(\theta_+(x)))$.
By Lemma \ref{lem-eq-erg-comp0} (i), we have $\theta_+(x)=\theta_+(x')$, which is $f(\theta_+(x))$.
This contradicts $A\cap f(A)=\emptyset$.

The second assertion of the lemma follows similarly. 
\end{proof}

By Lemma \ref{lem-inj}, for every Borel subset $A\subset Z_+$, the set
\[\bar{\theta}(A, t)\coloneqq \{ \, \bar{\theta}(y, t)\mid y\in A\, \}\]
is a Borel subset of $\calq$.
Similarly for every Borel subset $B\subset Z_-$, the set
\[\bar{\theta}(B, t^{-1})\coloneqq \{ \, \bar{\theta}(y, t^{-1})\mid y\in B\, \}\]
is a Borel subset of $\calq$.
Note that by Lemma \ref{lem-bar-theta} (iii), for every Borel subset $A\subset Z_+$, we have $\bar{\theta}(A, t)^{-1}=\bar{\theta}(t^{-1}A, t^{-1})$.

\begin{lem}\label{lem-section}
The following assertions hold:
\begin{enumerate}
\item[(i)] If $A\subset Z_+$ is a Borel subset on which $\sigma_+$ is injective, then $\bar{\theta}(A, t)$ is an $r$-section of $\calq$ and $r(\bar{\theta}(A, t))=\sigma_+(A)$.
\item[(ii)] If $B\subset Z_-$ is a Borel subset on which $\sigma_-$ is injective, then $\bar{\theta}(tB, t)$ is an $s$-section of $\calq$ and $s(\bar{\theta}(tB, t))=\sigma_-(B)$.
\end{enumerate}
Thus for the above $A$ and $B$, the set $\bar{\theta}(A\cap tB, t)$ belongs to $\llbracket \calq \rrbracket$.
\end{lem}

\begin{proof}
Assertion (i) follows from the equation $r(\bar{\theta}(y, t))=\sigma_+(y)$ for almost every $y\in Z_+$.
If $B\subset Z_-$ is a Borel subset on which $\sigma_-$ is injective, then $\bar{\theta}(B, t^{-1})$ is an $r$-section of $\calq$ because $r(\bar{\theta}(y, t^{-1}))=\sigma_-(y)$ for almost every $y\in Z_-$.
Assertion (ii) hence follows from the equation $\bar{\theta}(tB, t)=\bar{\theta}(B, t^{-1})^{-1}$. 
\end{proof}


\subsection{The Radon-Nikodym cocycle}\label{subsec-rn}

Let $(\calg, \mu)$ be a discrete measured groupoid.
We have the two equivalent measures $\mu_\calg^r$, $\mu_\calg^s$ on $\calg$ introduced in equation (\ref{eq-mu}).
We define the \textit{Radon-Nikodym cocycle} $\Delta \colon \calg \to \R_+^*$ of $(\calg, \mu)$ by $\Delta =d \mu_\calg^r/d \mu_\calg^s$, following \cite[Definition 2.1]{fm1}.

\begin{prop}[\ci{Proposition 2.2 and Corollary 2 in p.294}{fm1}]\label{prop-fm}
The map $\Delta$ is indeed a cocycle, and furthermore the following holds: 
For every $\phi \in \llbracket \calg \rrbracket$, we have
\[\frac{d\bar{\phi}_*\mu}{d\mu}(y)=\Delta(\phi y)\]
for almost every $y\in s(\phi)$.
\end{prop}

\begin{proof}
While this is given in \cite{fm1} for principal $\calg$ and follows from it, we supply the proof for the reader's convenience.
For every $\phi \in \llbracket \calg \rrbracket$ and every Borel subset $A\subset s(\phi)$,
\begin{align*}
(\bar{\phi}_*\mu)(A)=\mu(\bar{\phi}^{-1}(A))=\mu(r(\phi  A))=\mu_\calg^r(\phi A)=\int_{\phi  A}\frac{d\mu_\calg^r}{d\mu_\calg^s}\, d\mu_\calg^s =\int_A\frac{d\mu_\calg^r}{d\mu_\calg^s}(\phi y)\, d\mu(y).
\end{align*}
Therefore the latter assertion of the proposition holds.

We show that $\Delta$ is a cocycle.
For all $\phi, \psi \in \llbracket \calg \rrbracket$ with $s(\phi)=r(\psi)$ and for all bounded Borel functions $f$ on $A\coloneqq s(\psi)$, noticing $r(\psi \bar{\psi}(y))=y$ and $(\phi \psi)^-=\bar{\psi}\circ \bar{\phi}$, we have
\begin{align*}
&\int_Af(y)\Delta(\phi r(\psi y))\Delta(\psi y)\, d\mu(y)=\int_{\bar{\psi}^{-1}(A)}f(\bar{\psi}(y))\Delta(\phi y)\, d\mu(y)\\
&=\int_{\bar{\phi}^{-1}(\bar{\psi}^{-1}(A))}f(\bar{\psi}(\bar{\phi}(y)))\, d\mu(y)=\int_Af(y)\Delta(\phi \psi y)\, d\mu(y).
\end{align*}
Thus $\Delta(\phi \psi y)=\Delta(\phi r(\psi y))\Delta(\psi y)$ for almost every $y\in A$.
\end{proof}

Let $G$ be a countable group.
Let $E$ be a \textit{quasi-normal} subgroup of $G$, i.e., a subgroup of $G$ such that for every $g\in G$, the subgroup $E\cap gEg^{-1}$ is of finite index in both $E$ and $gEg^{-1}$.
We then have the \textit{modular homomorphism} $\bm{m} \colon G\to \Q_+^*$ defined by
\[\bm{m}(g)=\frac{[E: E\cap {}^g\!E]}{[{}^g\!E: E\cap {}^g\!E]}\]
for $g\in G$, where ${}^g\!E \coloneqq gEg^{-1}$.
This is indeed a homomorphism:
For all $g, h\in G$,
\begin{align*}
\bm{m}(g)=\frac{[E: E\cap {}^g\!E\cap {}^{gh}\!E]}{[{}^g\!E: E\cap {}^g\!E\cap {}^{gh}\!E]},\quad \bm{m}(h)=\frac{[{}^g\!E: {}^g\!E\cap {}^{gh}\!E]}{[{}^{gh}\!E: {}^g\!E \cap {}^{gh}\!E]}=\frac{[{}^g\!E: E\cap {}^g\!E\cap {}^{gh}\!E]}{[{}^{gh}\!E: E\cap {}^g\!E\cap {}^{gh}\!E]}
\end{align*}
and 
\[\bm{m}(g)\bm{m}(h)=[E: E\cap {}^g\!E\cap {}^{gh}\!E]/[{}^{gh}\!E: E\cap {}^g\!E\cap {}^{gh}\!E]=\bm{m}(gh).\]
The following generalizes a fact shown in the proof of \cite[Proposition B.2]{kida-bs}.

\begin{prop}\label{prop-rn}
With the notation in Lemma \ref{lem-ex-nor}, we assume that the action $G\c X$ preserves the probability measure $\mu$.
Let $(\calq, \zeta)$ be the quotient groupoid of $(X\rtimes G, \mu)$ by $X\rtimes E$, with the quotient homomorphism $\theta \colon X\rtimes G\to \calq$.
Note that $E$ is quasi-normal in $G$ and hence gives rise to the modular homomorphism $\bm{m}\colon G\to \Q_+^*$.
Then the Radon-Nikodym cocycle $\Delta$ of $(\calq, \zeta)$ satisfies the equation
\[\Delta(\theta(x, g))=\bm{m}(g)\]
for almost every $x\in X$ and every $g\in G$. 
\end{prop}

\begin{proof}
For every $a\in E$, $\theta(x, a)$ is a unit and hence $\Delta(\theta(x, a))=1=\bm{m}(a)$.
Since $E$ and $(t_n)$ generate $G$, it suffices to show that $\Delta(\theta(x, t_n))=\bm{m}(t_n)$ for every $n$ and almost every $x\in X$. 
We fix $n$ and put $t=t_n$ and $E_\pm =E_n^\pm$ for simplicity.
Let $\theta_\pm \colon (X, \mu)\to (Z_\pm, \zeta_\pm)$ be the ergodic decomposition maps of the p.m.p.\ actions $E_\pm \c (X, \mu)$, respectively.
Take finite Borel partitions $Z_-=\bigsqcup_k Z_-^k$ and $Z_+=\bigsqcup_l Z_+^l$ such that $\sigma_-$ is injective on each $Z_-^k$ and $\sigma_+$ is injective on each $Z_+^l$.

Fix $k$ and $l$.
Let $\phi \coloneqq \bar{\theta}(tZ_-^k\cap Z_+^l, t)$, where $\bar{\theta}(\cdot, t)$ is the symbol introduced right before Lemma \ref{lem-section}.
By Lemma \ref{lem-section}, we have $\phi \in \llbracket \calq \rrbracket$,
\[r(\phi)=\sigma_+(tZ_-^k\cap Z_+^l)\ \ \text{and}\ \ s(\phi)=\sigma_-(Z_-^k\cap t^{-1}Z_+^l).\]
If $s(\phi)$ is nonnull, then
\[\frac{\zeta(r(\phi))}{\zeta(s(\phi))}=\frac{[E: E_+]\zeta_+(tZ_-^k\cap Z_+^l)}{[E: E_-]\zeta_-(Z_-^k\cap t^{-1}Z_+^l)}=\frac{[E: E_+]}{[E: E_-]}=\bm{m}(t).\]
For every Borel subset $Y\subset s(\phi)$, partitioning $Z_-^k$ and $Z_+^l$ into finitely many Borel subsets further and applying the argument so far, we obtain $(\bar{\phi}_*\zeta)(Y)=\zeta(\bar{\phi}^{-1}(Y))=\bm{m}(t)\zeta(Y)$.
Therefore for almost every $x\in \theta_-^{-1}(Z_-^k\cap t^{-1}Z_+^l)$, we have
\[\Delta(\theta(tx, t))=\Delta(\phi \theta(x))=\frac{d\bar{\phi}_*\zeta}{d\zeta}(\theta(x))=\bm{m}(t),\]
where the second equation holds by Proposition \ref{prop-fm}.
This holds for all $k$ and $l$.
Thus the proposition follows.
\end{proof}



\section{Treeings and the induction}\label{sec-treeing}

Let $\calg$ be a discrete Borel groupoid.
Let $\Psi \subset \calg \setminus \calg^0$ be a Borel subset which is \textit{symmetric}, i.e., $\Psi =\Psi^{-1}$.
For each $x\in \calg^0$, we define the simplicial graph $\Psi(x)$ as follows:
Vertices of $\Psi(x)$ are elements of $x\calg$, and two vertices $g, h\in x\calg$ are joined by an edge if and only if $g^{-1}h\in \Psi$.
We mean by a \textit{$\Psi(x)$-edge} an edge of the graph $\Psi(x)$. 
For every $\gamma \in \calg$, the left-multiplication $s(\gamma)\calg \to r(\gamma)\calg$, $g\mapsto \gamma g$ induces the graph isomorphism $\Psi(s(\gamma))\to \Psi(r(\gamma))$.

\begin{defn}
A Borel subset $\Psi \subset \calg \setminus \calg^0$ is called a \textit{graphing} of $\calg$ if $\Psi$ is symmetric and the graph $\Psi(x)$ is connected for every $x\in \calg^0$. 
If $\Psi$ is a graphing of $\calg$ and $\Psi(x)$ is a tree for every $x\in \calg^0$, then $\Psi$ is called a \textit{treeing} of $\calg$.

Note that for a symmetric Borel subset $\Psi \subset \calg \setminus \calg^0$, the set of $x\in \calg^0$ such that $\Psi(x)$ is connected is Borel.
For a graphing $\Psi$ of $\calg$, the set of $x\in \calg^0$ such that $\Psi(x)$ is a tree is Borel.
\end{defn}

Let $\calg$ be a discrete Borel groupoid and let $\Psi \subset \calg \setminus \calg^0$ be a treeing of $\calg$.
Suppose that $\Psi$ has a Borel subset $\Psi_+$ such that $\Psi =\Psi_+\sqcup \Psi_+^{-1}$.
This condition is regarded as the treeing $\Psi$ being oriented.
Given a Borel subset $Y\subset \calg^0$, we construct a treeing of $\calg|_Y$, following Gaboriau's induction in \cite[Lemme II.8]{gaboriau}.

We define a function $d\colon \calg \to \N \cup \{ 0\}$ by
\[d(g)=\text{dist}_{\Psi(r(g))}(g, r(g)\calg Y)\]
for $g\in \calg$, where for $x\in \calg^0$, $\text{dist}_{\Psi(x)}$ is the graph metric on $\Psi(x)$ with each edge having distance $1$.

\begin{lem}\label{lem-d-inv}
The function $d$ is left-invariant, i.e., $d(g)=d(\gamma g)$ for all $\gamma, g\in \calg$ with $\gamma g$ defined.
\end{lem}

\begin{proof}
The left multiplication $g\mapsto \gamma g$ induces the graph isomorphism $\Psi(r(g))\to \Psi(r(\gamma))$ and the bijection $r(g)\calg Y\to r(\gamma)\calg Y$.
\end{proof}

We define a procedure to bring each vertex of $\Psi(x)$ to a vertex in $x\calg Y$ along a path in $\Psi(x)$, in a measurable way.
For $n\in \N \cup \{ 0\}$, write $\{  d=n \}$ for the set $\{ \, g\in \calg \mid d(g)=n\, \}$.
Then $\{  d=0 \} =\calg Y$.

Write $\Psi =\bigsqcup_{k=1}^\infty \psi_k$ as a disjoint union of $\psi_k\in \llbracket \calg \rrbracket$ such that for every odd $k$, we have $\psi_k\subset \Psi_+$ and $\psi_{k+1}=\psi_k^{-1}$.
Hence $\Psi_+$ is the disjoint union of all $\psi_k$ with odd $k$.
For each $n\in \N$, we define a map
\[f_n\colon \{  d=n \} \to \{  d=n-1 \}\]
as follows.
Given $g\in \{  d=n \}$, we choose the minimal $k\in \N$ such that $s(g)\in r(\psi_k)$ and $g\psi_k\in \{ d=n-1 \}$, and define $f_n(g)=g\psi_k$.
Using this map $f_n$, we further define a map $f\colon \calg \to \calg Y$ by
\[f(g)=\begin{cases}
g & \text{if}\ g\in \calg Y,\\
(f_1\circ f_2\circ \cdots \circ f_n)(g) & \text{if}\ g\in \{  d=n \} \ \text{and}\ n\in \N.
\end{cases}\] 
The map $f$ is Borel.

\begin{lem}\label{lem-f-eq}
The maps $f_n$ and $f$ are equivariant with respect to the left-multiplication of $\calg$.
\end{lem}

\begin{proof}
It suffices to check the equivariance of $f_n$.
Suppose $g\in \{  d=n  \}$ and $f_n(g)=g\psi_k$ with $k\in \N$.
For all $\gamma \in \calg r(g)$, by Lemma \ref{lem-d-inv}, $n=d(g)=d(\gamma g)$ and $n-1=d(g\psi_k)=d(\gamma g\psi_k)$.
By the minimality of $k$ and the definition of $f_n$, we have $f_n(\gamma g)=\gamma g\psi_k$.
\end{proof}

We define the sets
\[\Psi_{0, +} =\{ \, g^{-1}f_n(g)\mid g\in \{ d=n\},\, n\in \N \, \}\ \  \text{and}\ \ \Psi_0 =\Psi_{0, +}\cup \Psi_{0, +}^{-1},\]
which are Borel subsets of $\Psi$.
Note that $\Psi_{0, +}$ is not necessarily contained in $\Psi_+$.

\begin{rem}\label{rem-unique-0+}
For every $x\in \calg^0\setminus Y$, there exists a unique element of $\Psi_{0, +}$ whose range is $x$.
In fact, that unique element is written as $g^{-1}f_{d(g)}(g)$ for an arbitrary $g \in \calg x$.
By Lemmas \ref{lem-d-inv} and \ref{lem-f-eq}, $g^{-1}f_{d(g)}(g)$ depends only on $x$ and is independent of the choice of $g$.
\end{rem}

Given a $\Psi(x)$-edge $(g, g\gamma)$ with $\gamma \in \Psi \setminus \Psi_0$, we slide the vertices $g$, $g\gamma$ along the oriented edges in $\Psi_{0, +}$.
More precisely, putting $n=d(g)$, we have the geodesic in the graph $\Psi(x)$:
\[g,\ f_n(g),\ (f_{n-1}\circ f_n)(g),\ldots,\ (f_1\circ \cdots \circ f_n)(g)=f(g)\in x\calg Y.\]
Let us denote this geodesic by $l_g$.  
We also have the geodesic $l_{g\gamma}$ from $g\gamma$ to a vertex in $x\calg Y$ defined similarly. 
We will join the two vertices $f(g)$, $f(g\gamma )$ by an edge and obtain a graph with $x\calg Y$ being the set of vertices.
This graph will be shown to be a tree.

We put $\Psi_1=\Psi \setminus \Psi_0$.
For $\gamma \in \Psi_1$ and $g\in \calg$ with $g\gamma$ defined, we define $J(\gamma)\in \calg|_Y$ by
\[J(\gamma)=f(g)^{-1}f(g\gamma).\]
By Lemma \ref{lem-f-eq}, the right hand side does not depend on the choice of $g$.
The element $J(\gamma)$ defines an edge joining the two vertices $f(g)$, $f(g\gamma)=f(g)J(\gamma)$.

\begin{lem}\label{lem-J}
For every $\gamma \in \Psi_1$, we have $J(\gamma)\not\in Y$ and $J(\gamma^{-1})=J(\gamma)^{-1}$.
\end{lem}

\begin{proof}
Pick $\gamma \in \Psi_1$ and $g\in \calg$ with $g\gamma$ defined, and put $h=g\gamma$.
The latter assertion follows from $J(\gamma^{-1})=f(h)^{-1}f(h\gamma^{-1})=J(\gamma)^{-1}$.

To verify $J(\gamma)\not\in Y$, we suppose $f(g)=f(g\gamma)$ toward contradiction.
Since $\gamma \not\in \Psi_0$, we have neither $f_{d(g)}(g)=g\gamma$ nor $f_{d(g\gamma)}(g\gamma)=g$.
Therefore the geodesic $l_g$ does not contain $g\gamma$, and $l_{g\gamma}$ does not contain $g$.
By our assumption, $l_g$ and $l_{g\gamma}$ have the same terminus.
Thus there exists a simple loop consisting of the edge $(g, g\gamma)$, $l_g$ and $l_{g\gamma}$.
This contradicts $\Psi(r(g))$ being a tree.
\end{proof}

By the last lemma, we obtain the map $J\colon \Psi_1\to \calg|_Y\setminus Y$.

\begin{lem}\label{lem-J-inj}
The map $J\colon \Psi_1\to \calg|_Y\setminus Y$ is injective.
\end{lem}

\begin{proof}
Pick $\gamma, \delta \in \Psi_1$ and suppose that $f(g)^{-1}f(g\gamma)=f(h)^{-1}f(h\delta)$ for some $g, h\in \calg$ with $g\gamma$ and $h\delta$ defined.
We put $x=r(g)$, $y=r(h)$ and
\[\eta \coloneqq f(h)f(g)^{-1}=f(h\delta )f(g\gamma )^{-1}\in y\calg x.\]
Since $f(g)\neq f(g\gamma)$ by Lemma \ref{lem-J}, we have the geodesic $l_1$ in $\Psi(x)$ through the vertices $f(g), g, g\gamma, f(g\gamma)$ in this order that connects the reverse of $l_g$, the edge $(g, g\gamma )$, and $l_{g\gamma}$.
Similarly we have the geodesic $l_2$ in $\Psi(y)$ through the vertices $f(h), h, h\delta,  f(h\delta)$ in this order.
Multiplying $l_1$ by $\eta$ from the left, we obtain the geodesic $\eta l_1$ in $\Psi(y)$ through the vertices 
\[\eta f(g)=f(h),\, \eta g,\, \eta g\gamma,\, \eta f(g\gamma)=f(h\delta)\]
in this order.
The two geodesics $l_2$, $\eta l_1$ must be equal because they have the same origin and terminus.
Any two successive vertices in $l_2$ between $f(h)$ and $h$ and between $h\delta$ and $f(h\delta)$ are joined by an element of $\Psi_0$, and the two vertices $h$, $h\delta$ are joined by an element of $\Psi_1$.
A similar description for $\eta l_1$ holds, and it follows that $(h, h\delta)=(\eta g, \eta g\gamma)$.
Thus $\gamma =\delta$.
\end{proof}

\begin{prop}\label{prop-theta-tree}
The image $\Theta \coloneqq J(\Psi_1)$ is a treeing of $\calg|_Y$.
\end{prop}

\begin{proof}
By Lemma \ref{lem-J}, $\Theta$ is symmetric.
We fix $x\in Y$.
First we show that the graph $\Theta(x)$ is connected.
Pick $g, h\in x\calg Y$.
Since $\Psi(x)$ is connected, we have the geodesic through the vertices $g=g_0, g_1, g_2,\ldots, g_n=h$, where any two successive vertices form a $\Psi(x)$-edge.
We then have the path in $\Theta(x)$ through the vertices $g=f(g), f(g_1), f(g_2),\ldots, f(g_n)=h$, where $(f(g_i), f(g_{i+1}))$ is a $\Theta(x)$-edge if $(g_i, g_{i+1})$ is a $\Psi_1(x)$-edge, and otherwise we have $f(g_i)=f(g_{i+1})$.
Thus $g$ and $h$ are joined by a path in $\Theta(x)$.

To show that $\Theta(x)$ is a tree, we prepare the following:

\begin{lem}\label{lem-g-h}
For every $\Theta(x)$-edge $(g_0, g_1)$, there exists a $\Psi_1(x)$-edge $(h_0, h_1)$ such that $f(h_0)=g_0$ and $f(h_1)=g_1$.
\end{lem}

\begin{proof}
We have $g_1=g_0J(\gamma)$ for some $\gamma \in \Psi_1$ and have $J(\gamma)=f(a)^{-1}f(a\gamma)$ for some or any $a\in \calg$ with $a\gamma$ defined.
The product $g_0f(a)^{-1}$ is defined. 
We put $h_0=g_0f(a)^{-1}a$ and $h_1=h_0\gamma$.
By the equivariance of $f$ (Lemma \ref{lem-f-eq}), we have $f(h_0)=g_0f(a)^{-1}f(a)=g_0$ and $f(h_1)=g_0f(a)^{-1}f(a\gamma)=g_0J(\gamma)=g_1$.
\end{proof}

\begin{rem}\label{rem-unique}
In Lemma \ref{lem-g-h}, uniqueness of the edge $(h_0, h_1)$ also holds.
This is proved via argument similar to the proof of Lemma \ref{lem-J-inj} as follows (though we do not use this in the sequel):
Let $(h_0', h_1')$ be another $\Psi_1(x)$-edge such that $f(h_0')=g_0$ and $f(h_1')=g_1$.
Then we have the closed path in $\Psi(x)$ through the vertices
\[h_0,\, f(h_0)=f(h_0'),\, h_0',\, h_1',\, f(h_1')=f(h_1),\, h_1,\, h_0\]
in this order, where for each $h\in \{ h_0, h_1, h_0', h_1'\}$, $h$ and $f(h)$ are joined by $l_h$.
In this path, $(h_0', h_1')$ and $(h_1, h_0)$ are the only $\Psi_1(x)$-edges and therefore $h_0=h_0'$ and $h_1=h_1'$.
\end{rem}

We turn to the proof of $\Theta(x)$ being a tree.
Suppose toward a contradiction that $\Theta(x)$ has a simple loop $l$ through vertices $g_0, g_1, g_2,\ldots, g_n=g_0$. 
We will construct a simple loop in $\Psi(x)$.
By Lemma \ref{lem-g-h}, there exists a $\Psi_1(x)$-edge $(h_0, h_1)$ such that $f(h_0)=g_0$ and $f(h_1)=g_1$.
By the same lemma again, for every $i=1,\ldots, n-1$, there exists a $\Psi_1(x)$-edge $(k_i, h_{i+1})$ such that $f(k_i)=g_i$ and $f(h_{i+1})=g_{i+1}$.

For every $i=1,\ldots, n-1$, the two geodesics $l_{h_i}$, $l_{k_i}$ have the same terminus $g_i$.
They consist of $\Psi_0(x)$-edges and do not contain the edges $(k_{i-1}, h_i)$, $(k_i, h_{i+1})$, where $k_0\coloneqq h_0$.
We thus have the geodesic in $\Psi(x)$ through $k_{i-1}, h_i, k_i, h_{i+1}$ in this order.
Similarly we have the geodesic in $\Psi(x)$ through $k_{n-1}, h_n, h_0, h_1$ in this order.
These geodesics form a simple loop in $\Psi(x)$, and this contradicts $\Psi(x)$ being a tree.
\end{proof}

We suppose that there is a $\sigma$-finite measure $\mu$ on $\calg^0$ preserved by $\calg$.
For a symmetric Borel subset $S \subset \calg$, we define its \textit{cost} as the value
\[C_\mu (S)=\mu_\calg^r(S)/2.\]

\begin{prop}\label{prop-theta-cost}
The isomorphism $J\colon \Psi_1\to \Theta$ preserves $\mu_\calg^r$.
Thus $C_\mu(\Theta)=C_\mu(\Psi_1)$.
\end{prop}

\begin{proof}
By dividing $\Psi_1$ into countably many Borel subsets, it suffices to show that if $D\subset \Psi_1$ is a Borel subset such that $D$ and $J(D)$ are $r$-sections of $\calg$, then the map $J\colon D\to J(D)$ preserves $\mu_\calg^r$.
By the definition of $\mu_\calg^r$, this is equivalent to saying that the map $J_1\colon r(D)\to r(J(D))$ defined by $r(\gamma)\mapsto r(J(\gamma))$ for $\gamma \in D$ preserves $\mu$.
The map $f\colon \calg \to \calg Y$ preserves the range and satisfies $J(\gamma)=f(r(\gamma))^{-1}f(\gamma)$ for every $\gamma \in \Psi_1$.
Then $J_1(x)=s(f(x))$ for every $x\in r(D)$ and thus $J_1$ preserves $\mu$.
\end{proof}


\section{HNN extensions and treeings}\label{sec-HNN}

Throughout this section, we keep the following notation:
Let $G$ be the HNN extension
\[G=\langle \, E, \, t\mid \forall a\in E_-\ \ tat^{-1}=\tau(a)\, \rangle,\]
where $E$ is a countable group and $\tau \colon E_-\to E_+$ is an isomorphism between finite-index normal subgroups $E_-$, $E_+$ of $E$.
We set $p=[E: E_-]$ and $q=[E: E_+]$.

Let $G\c (X, \mu)$ be a p.m.p.\ action and suppose that there exist $E$-equivariant Borel maps $X\to E/E_-$ and $X\to E/E_+$.
We set $\calg =X\rtimes G$ and $\cale =X\rtimes E$.
Then $\cale \vartriangleleft \calg$ by Lemma \ref{lem-ex-nor}.
Let $(\calq, \zeta)$ be the quotient groupoid of $(\calg, \mu)$ by $\cale$ with the quotient homomorphism $\theta \colon \calg \to \calq$ and set $Z=\calq^0$.

Let $\theta_\pm \colon (X, \mu)\to (Z_\pm, \zeta_\pm)$ be the ergodic decomposition maps of the p.m.p.\ actions $E_\pm \c (X, \mu)$, respectively.
We have the measure-preserving maps $\sigma_\pm \colon Z_\pm \to Z$ such that $\theta =\sigma_\pm \circ \theta_\pm$.
Recall that the map $X\to \calq$, $x\mapsto \theta(x, t)$ is $E_+$-invariant by Lemma \ref{lem-eq-erg-comp0} (i) and hence induces the map $Z_+\to \calq$, which we denote by $y\mapsto \bar{\theta}(y, t)$.
Similarly the map $X\to \calq$, $x\mapsto \theta(x, t^{-1})$ induces the map $Z_-\to \calq$, $y\mapsto \bar{\theta}(y, t^{-1})$. 
We define 
\[\Phi_+\coloneqq \bar{\theta}(Z_+, t )=\bar{\theta}(Z_-, t^{-1})^{-1},\]
where the last equation follows from Lemma \ref{lem-bar-theta} (iii).

Under this setting, in Subsection \ref{subsec-treeing}, we prove that $\Phi_+$ defines an oriented treeing of $\calq$.
In Subsection \ref{subsec-induced}, we construct the Maharam extension $(\tilde{\calq}, \tilde{\zeta})$ of $(\calq, \zeta)$ and induce from this treeing of $\calq$ a treeing of a certain restriction of $\tilde{\calq}$.
We then obtain a treeing of the kernel of the Radon-Nikodym cocycle of $(\calq, \zeta)$ and compute its cost.

\subsection{Being a treeing}\label{subsec-treeing}

We start with a few preliminary observations.

\begin{lem}
The sets $\Phi_+$ and $\Phi_+^{-1}$ are disjoint.
\end{lem}

\begin{proof}
Otherwise there would exist nonnull Borel subsets $A\subset Z_+$ and $B\subset Z_-$ such that $\bar{\theta}(A, t)$ and $\bar{\theta}(B, t^{-1})$ have the nonnull intersection.
By replacing $A$ and $B$ into smaller nonnull subsets, we may assume that $\sigma_+$ is injective on $A$ and $\sigma_-$ is injective on $B$.
By a further replacement of $A$ and $B$ into smaller nonnull subsets, we obtain an isomorphism $f\colon A\to B$ such that $\bar{\theta}(y, t)=\bar{\theta}(f(y), t^{-1})$ for all $y\in A$ and $qf_*(\zeta_+|_A)=p\zeta_-|_B$.
Here the last equation holds because $q(\sigma_+)_*(\zeta_+|_A)=\zeta|_{\sigma_+(A)}$ and $p(\sigma_-)_*(\zeta_-|_B)=\zeta|_{\sigma_-(B)}$.

For almost every $x\in \theta_+^{-1}(A)$, the following holds:
For almost every $x'\in \theta_-^{-1}(f(\theta_+(x)))$, we have
\[\theta(x, t)=\bar{\theta}(\theta_+(x), t)=\bar{\theta}(f(\theta_+(x)), t^{-1})=\theta(x', t^{-1}).\]
The range of the both sides is $\theta(x)=\theta(x')$.
Hence there exists $a\in E$ such that $\theta_+(x)=\theta_+(ax')$ and 
\[\theta(x', t^{-1})=\theta(x, t)=\theta(ax', t)=\theta(x', a^{-1})\theta(ax', t)=\theta(x', a^{-1}t),\]
where the second equation holds by Lemma \ref{lem-eq-erg-comp0} (i).
Thus 
\[\theta(x', a^{-1}t)^{-1}\theta(x', t^{-1})=\theta(t^{-1}ax', t^{-1}at^{-1})\]
is a unit.
This contradicts $t^{-1}at^{-1}\not\in E$.
\end{proof}

We set $\Phi =\Phi_+\sqcup \Phi_+^{-1}$.
For each $z\in Z$, we have the graph $\Phi(z)$ oriented with respect to this decomposition.
More precisely for any two adjacent vertices $\gamma, \delta \in z\calq$ of $\Phi(z)$, when $\gamma^{-1}\delta\in \Phi_+$, we call the edge between $\gamma$ and $\delta$ \textit{outgoing} from $\gamma$ into $\delta$ or \textit{incoming} into $\delta$ from $\gamma$.
Fix subsets $S_\pm\subset E$ of representatives of cosets in $E/E_\pm$, respectively.

\begin{lem}\label{lem-Phi-edge}
For almost every $z\in Z$, for every vertex $\gamma \in z\calq$ of $\Phi(z)$, the following assertions hold:
\begin{enumerate}
\item[(i)] For every edge outgoing from $\gamma$, there exists a unique $u\in \sigma_+^{-1}(s(\gamma))$ such that the edge is incoming into $\gamma \bar{\theta}(u, t)$.
\item[(ii)] For every edge incoming into $\gamma$, there exists a unique $v\in \sigma_-^{-1}(s(\gamma))$ such that the edge is outgoing from $\gamma \bar{\theta}(v, t^{-1})$.
\end{enumerate} 
\end{lem}

\begin{proof}
The existence and uniqueness follow from Lemmas \ref{lem-bar-theta} and \ref{lem-inj}, respectively.
\end{proof}

\begin{lem}\label{lem-Phi-graphing}
The set $\Phi$ is a graphing of $\calq$.
\end{lem}

\begin{proof}
By the class-surjectivity of $\theta$, it suffices to show that for almost every $x\in X$ and every $g\in G$, $\theta(x, g)$ is written as a $\Phi$-word, i.e., a product of finitely many elements of $\Phi$.
Since $\theta(x, ga)=\theta(x, g)\theta(g^{-1}x, a)=\theta(x, g)$ for every $a\in E$, it suffices to show this for $g$ of the form $g=a_1t^{\ve_1}a_2t^{\ve_2}\cdots a_nt^{\ve_n}$ with $a_i\in E$ and $\ve_i \in \{ \pm 1\}$.

We show this by induction on $n$.
When $n=1$, we have
\[\theta(x, a_1t^{\ve_1})=\theta(x, a_1)\theta(a_1^{-1}x, t^{\ve_1})=\theta(a_1^{-1}x, t^{\ve_1}),\]
which is $\bar{\theta}(\theta_+(a_1^{-1}x), t)$ if $\ve_1=1$, and which is $\bar{\theta}(\theta_-(a_1^{-1}x), t^{-1})$ if $\ve_1=-1$.
It follows that $\theta(x, a_1t^{\ve_1})$ belongs to $\Phi$.

For general $g$, we have
\[\theta(x, g)=\theta(x, g_{n-1})\theta(g_{n-1}^{-1}x, a_nt^{\ve_n}),\]
where $g_{n-1}\coloneqq a_1t^{\ve_1}\cdots a_{n-1}t^{\ve_{n-1}}$.
By the induction hypothesis, $\theta(x, g_{n-1})$ is written as a $\Phi$-word.
We have
\[\theta(g_{n-1}^{-1}x, a_nt^{\ve_n})=\theta(g_{n-1}^{-1}x, a_n)\theta(a_n^{-1}g_{n-1}^{-1}x, t^{\ve_n})=\theta(a_n^{-1}g_{n-1}^{-1}x, t^{\ve_n}),\]
which is $\bar{\theta}(\theta_+(a_n^{-1}g_{n-1}^{-1}x), t)$ if $\ve_n=1$, and which is $\bar{\theta}(\theta_-(a_n^{-1}g_{n-1}^{-1}x), t^{-1})$ if $\ve_n=-1$.
Thus $\theta(x, g)$ is a $\Phi$-word.
\end{proof}

\begin{lem}\label{lem-word}
For $i=1,\ldots, n$, let $A_i\subset Z_+$ be a Borel subset such that $\varphi_i\coloneqq \bar{\theta}(A_i, t)$ belongs to $\llbracket \calq \rrbracket$.
Let $\varphi =\varphi_1^{\ve_1}\cdots \varphi_n^{\ve_n}$ with $\ve_i\in \{ \pm 1\}$.
Then for almost every $x\in \theta^{-1}(r(\varphi))$, there exist $a_i\in S_{\ve_i}$ (where $S_1\coloneqq S_+$ and $S_{-1}\coloneqq S_-$) such that
\begin{enumerate}
\item[(1)] we have $\theta(x)\varphi =\theta(x, a_1t^{\ve_1}a_2t^{\ve_2}\cdots a_nt^{\ve_n})$ and
\item[(2)] for every $k=1,\ldots, n$,
\[s(\theta(x)\varphi_1^{\ve_1}\cdots \varphi_{k-1}^{\ve_{k-1}})\varphi_k^{\ve_k}=\theta(g_{k-1}^{-1}x, a_kt^{\ve_k}),\]
where $g_k\coloneqq a_1t^{\ve_1}a_2t^{\ve_2}\cdots a_kt^{\ve_k}$ and $g_0\coloneqq e$.
\end{enumerate}
\end{lem}

\begin{proof}
First note that condition (1) follows from condition (2).
Indeed by condition (2),
\begin{align*}
\theta(x)\varphi_1^{\ve_1}\varphi_2^{\ve_2}\cdots \varphi_n^{\ve_n} = \theta(x, a_1t^{\ve_1})\theta(g_1^{-1}x, a_2t^{\ve_2})\cdots \theta(g_{n-1}^{-1}x, a_nt^{\ve_n})=\theta(x, g_n).
\end{align*}

We verify the existence of $a_i$ satisfying condition (2) by induction on $n$.
If $\ve_1=1$, then $\theta(x)\varphi_1 =\bar{\theta}(y, t)$ for some $y\in A_1$ and there exists a unique $a_1\in S_+$ such that $\theta_+(x)=a_1y$ and hence $\theta_+(a_1^{-1}x)=y$.
Thus
\[\theta(x)\varphi_1 =\bar{\theta}(y, t)=\theta(a_1^{-1}x, t)=\theta(x, a_1)\theta(a_1^{-1}x, t)=\theta(x, a_1t)\]
and the proof completes.

If $\ve_1=-1$, then $\theta(x)\varphi_1^{-1} =\bar{\theta}(y, t)^{-1}$ for some $y\in A_1$ and there exists a unique $a_1\in S_-$ such that $\theta_-(x)=a_1t^{-1}y$ and hence $\theta_-(a_1^{-1}x)=t^{-1}y$.
Thus
\[\theta(x)\varphi_1^{-1} =\bar{\theta}(y, t)^{-1}=\bar{\theta}(t^{-1}y, t^{-1})=\theta(a_1^{-1}x, t^{-1})=\theta(x, a_1)\theta(a_1^{-1}x, t^{-1})=\theta(x, a_1t^{-1})\]
and the proof completes.

We verify the existence of $a_i$ for general $\varphi =\varphi_1^{\ve_1}\cdots \varphi_n^{\ve_n}$.
By the induction hypothesis, for almost every $x\in \theta^{-1}(r(\varphi))$, there exists $a_i\in S_{\ve_i}$ with $i=1,\ldots, n-1$ such that for every $k=1,\ldots, n-1$,
\[s(\theta(x)\varphi_1^{\ve_1}\cdots \varphi_{k-1}^{\ve_{k-1}})\varphi_k^{\ve_k} =\theta(g_{k-1}^{-1}x, a_kt^{\ve_k}).\]
Then $\theta(x)\varphi =\theta(x, g_{n-1})\varphi_n^{\ve_n}$.
The rest of the proof is similar to that for the case of $n=1$:
If $\ve_n=1$, then $\theta(x, g_{n-1})\varphi_n=\theta(x, g_{n-1})\bar{\theta}(y, t)$ for some $y\in A_n$ and there exists a unique $a_n\in S_+$ such that $\theta_+(g_{n-1}^{-1}x)=a_ny$ and hence $\theta_+(a_n^{-1}g_{n-1}^{-1}x)=y$.
Thus
\begin{align*}
&\theta(x, g_{n-1})\varphi_n=\theta(x, g_{n-1})\bar{\theta}(y, t)=\theta(x, g_{n-1})\theta(a_n^{-1}g_{n-1}^{-1}x, t)\\
&=\theta(x, g_{n-1})\theta(g_{n-1}^{-1}x, a_n)\theta(a_n^{-1}g_{n-1}^{-1}x, t)=\theta(x, g_{n-1}a_nt)
\end{align*}
and the proof completes.

If $\ve_n=-1$, then $\theta(x, g_{n-1})\varphi_n^{-1}=\theta(x, g_{n-1})\bar{\theta}(y, t)^{-1}$ for some $y\in A_n$ and there exists a unique $a_n\in S_-$ such that $\theta_-(g_{n-1}^{-1}x)=a_nt^{-1}y$ and hence $\theta_-(a_n^{-1}g_{n-1}^{-1}x)=t^{-1}y$.
Thus
\begin{align*}
&\theta(x, g_{n-1})\varphi_n^{-1}=\theta(x, g_{n-1})\bar{\theta}(y, t)^{-1}=\theta(x, g_{n-1})\bar{\theta}(t^{-1}y, t^{-1})\\
&=\theta(x, g_{n-1})\theta(a_n^{-1}g_{n-1}^{-1}x, t^{-1})=\theta(x, g_{n-1})\theta(g_{n-1}^{-1}x, a_n)\theta(a_n^{-1}g_{n-1}^{-1}x, t^{-1})\\
&=\theta(x, g_{n-1}a_nt^{-1})
\end{align*}
and the proof completes.
\end{proof}

Recall that the HNN extension $G$ admits the Bass-Serre tree $T$ such that its vertex set is $G/E$ and for every $g\in G$, the two vertices $gE$, $gtE$ are joined by an oriented edge outgoing from $gE$ into $gtE$ (\cite[I.5]{serre}).
This is used in the proof of the following:

\begin{lem}\label{lem-Phi-treeing}
As in Lemma \ref{lem-word}, for $i=1,\ldots, n$, let $A_i\subset Z_+$ be a Borel subset such that $\varphi_i\coloneqq \bar{\theta}(A_i, t)$ belongs to $\llbracket \calq \rrbracket$.
Let $\varphi =\varphi_1^{\ve_1}\cdots \varphi_n^{\ve_n}$ with $\ve_i\in \{ \pm 1\}$.
Suppose that $r(\varphi)$ is nonnull and for every $z\in r(\varphi)$, we have $z\varphi =z$.
Then there exist $i\in \{ 1,\ldots, n-1\}$ and a nonnull Borel subset $A\subset Z$ such that $z\varphi_i^{\ve_i}\varphi_{i+1}^{\ve_{i+1}}=z$ for all $z\in A$.
 \end{lem}

\begin{proof}
By condition (1) in Lemma \ref{lem-word}, there exist a nonnull Borel subset $D\subset \theta^{-1}(r(\varphi))$ and $a_i\in S_{\ve_i}$ such that for every $x\in D$, we have
\[\theta(x)\varphi =\theta(x, a_1t^{\ve_1}a_2t^{\ve_2}\cdots a_nt^{\ve_n}),\]
which is $\theta(x)$ by our assumption.
It follows that $a_1t^{\ve_1}a_2t^{\ve_2}\cdots a_nt^{\ve_n}\in E$.
The sequence
\[E,\ a_1t^{\ve_1}E,\ a_1t^{\ve_1}a_2t^{\ve_2}E,\ \ldots,\ a_1t^{\ve_1}a_2t^{\ve_2}\cdots a_nt^{\ve_n}E=E\]
forms a cycle in the Bass-Serre tree $T$, which has to contain a backtracking.
There exists $i\in \{ 1,\ldots, n-1\}$ such that $a_{i}t^{\ve_{i}}a_{i+1}t^{\ve_{i+1}}\in E$.
By condition (2) in Lemma \ref{lem-word}, for almost every $x\in D$,
\begin{align*}
s(\theta(x)\varphi_1^{\ve_1}\cdots \varphi_{i-1}^{\ve_{i-1}})\varphi_{i}^{\ve_{i}}\varphi_{i+1}^{\ve_{i+1}}&=\theta(g_{i-1}^{-1}x, a_{i}t^{\ve_i})\theta(g_{i}^{-1}x, a_{i+1}t^{\ve_{i+1}})\\
&=\theta(g_{i-1}^{-1}x, a_{i}t^{\ve_{i}}a_{i+1}t^{\ve_{i+1}}),
\end{align*}
where $g_k\coloneqq a_1t^{\ve_1}a_2t^{\ve_2}\cdots a_kt^{\ve_k}$ and $g_0\coloneqq e$.
Since $a_{i}t^{\ve_{i}}a_{i+1}t^{\ve_{i+1}}\in E$, the right hand side is a unit.
Let $A\subset Z$ be a Borel subset such that $\theta^{-1}(A)$ is equal to the saturation $[D]_\cale$.
Then $A$ is nonnull and for almost every $z\in A$, $s(z\varphi_1^{\ve_1}\cdots \varphi_{i-1}^{\ve_{i-1}})\varphi_{i}^{\ve_{i}}\varphi_{i+1}^{\ve_{i+1}}$ is a unit.
\end{proof}

By Lemmas \ref{lem-Phi-graphing} and \ref{lem-Phi-treeing}, we obtain:

\begin{cor}
The set $\Phi$ is a treeing of $\calq$.
\end{cor}


\subsection{Maharam extensions and induced treeings}\label{subsec-induced}

Let $\bm{m}\colon G\to \Q^*_+$ be the modular homomorphism associated to the quasi-normal subgroup $E$ of $G$.
Let $\Delta \colon \calq \to \R^*_+$ be the Radon-Nikodym cocycle of $(\calq, \zeta)$.
By Proposition \ref{prop-rn}, $\Delta(\theta(x, t))=\bm{m}(t)=q/p$ for almost every $x\in \calg^0$.
We define $L$ as the subgroup $(q/p)^\Z$ of $\R^*_+$, which is the range of $\Delta$.
Let $L$ act on $L$ by multiplication (from the left), and let $\tilde{L}\coloneqq L\rtimes L$ be the associated translation groupoid.
Its product is given by $(l, u)(u^{-1}l, v)=(l, uv)$ for $l, u, v\in L$.

The product set $\calq \times \tilde{L}$ is regarded as the groupoid on $\calq^0\times L$ such that the range and source of $(g, \gamma)\in \calq \times \tilde{L}$ are $(r(g), r(\gamma))$ and $(s(g), s(\gamma))$, respectively, the product is given by $(g, \gamma)(h, \delta)=(gh, \gamma \delta)$ (when both the products $gh$ and $\gamma \delta$ are defined), and the inverse is given by $(g, \gamma)^{-1}=(g^{-1}, \gamma^{-1})$.

We now define the \textit{Maharam extension} $(\tilde{\calq}, \tilde{\zeta})$ of $(\calq, \zeta)$.
Let $\tilde{\calq}$ be the subgroupoid of $\calq \times \tilde{L}$ defined by
\[\tilde{\calq}= \{ \, (g, (l, \Delta(g)))\in \calq \times \tilde{L}\mid g\in \calq,\, l\in L\, \}.\]
We endow $\tilde{L}^0=L$ with the measure $\eta$ such that the point $(q/p)^n$ has measure $(q/p)^{-n}$ for each $n\in \Z$.
We set $W=\tilde{\calq}^0=\calq^0\times L$ and endow $W$ with the measure $\tilde{\zeta}\coloneqq \zeta \times \eta$.
Then $\tilde{\calq}$ preserves $\tilde{\zeta}$.

In the previous subsection, we constructed the treeing $\Phi =\Phi_+\sqcup \Phi_+^{-1}$ of $\calq$.
We set
\[\tilde{\Phi}_+=\{ \, (g, (l, \Delta(g)))\mid g\in \Phi_+,\, l\in L\, \}\]
and set $\tilde{\Phi}=\tilde{\Phi}_+\sqcup \tilde{\Phi}_+^{-1}$.

\begin{lem}\label{lem-tilde-Phi-treeing}
The set $\tilde{\Phi}$ is a treeing of $\tilde{\calq}$.
Moreover the following assertions hold:
\begin{enumerate}
\item[(i)] For all $w\in W$, the set $\tilde{\Phi}_+\cap w\tilde{\calq}$ consists of exactly $q$ elements.
\item[(ii)] For all $w\in W$, the set $\tilde{\Phi}_+^{-1}\cap w\tilde{\calq}$ consists of exactly $p$ elements.
\end{enumerate}
\end{lem}

\begin{proof}
For all $z\in \calq^0$ and $l\in L$, we have the bijection $z\calq \to (z, l)\tilde{\calq}$, $g\mapsto (g, (l, \Delta(g)))$, which induces a graph isomorphism from $\Phi(z)$ onto $\tilde{\Phi}(z, l)$.
The former assertion of the lemma follows.
For assertion (i), it suffices to show that for all $z\in \calq^0$, the set $\Phi_+\cap z\calq$ consists of exactly $q$ elements.
This follows from Lemma \ref{lem-Phi-edge} (i).
Similarly assertion (ii) follows from Lemma \ref{lem-Phi-edge} (ii).
\end{proof}

We set $N= \ker \bm{m}$ and $\caln = X\rtimes N$.
Then $\cale \vartriangleleft \caln$ (by Corollary \ref{cor-shg}), and the quotient groupoid $\caln /\cale$ is naturally identified with a subgroupoid of $\calq =\calg /\cale$ (by the universal property (3) in Theorem \ref{thm-quotient}). 
Moreover $\caln /\cale =\ker \Delta$, and $\caln /\cale$ is identified with the restriction $\tilde{\calq}|_{\calq^0\times \{ 1\}}$ (by ignoring the coordinate in $\tilde{L}$).
Applying the induction discussed in Section \ref{sec-treeing} to the treeing $\tilde{\Phi}$, we will obtain a treeing of $\tilde{\calq}|_{\calq^0\times \{ 1\}}\simeq \caln /\cale$.

If $p=q$, then $N=G$ and $L=\{ 1\}$.
Therefore the induction is needless.
In the rest of this subsection (until Theorem \ref{thm-cost-infty}), we suppose $p\neq q$.

For $n\in \Z$, we set $W_n=\calq^0\times \{ (q/p)^n\}\subset W$.
Write $\Phi =\bigsqcup_{k=1}^\infty \varphi_k$ as the disjoint union of $\varphi_k\in \llbracket \calq \rrbracket$ such that for every odd $k$, $\varphi_k\subset \Phi_+$ and $\varphi_{k+1}=\varphi_k^{-1}$.
Define $\tilde{\varphi}_k\subset \tilde{\Phi}$ by
\[\tilde{\varphi}_k=\{ \, (g, (l, \Delta(g)))\mid g\in \varphi_k,\, l\in L\, \}.\]
Then $\tilde{\Phi}=\bigsqcup_{k=1}^\infty \tilde{\varphi}_k$.
We define a function $d\colon \tilde{\calq}\to \N \cup \{ 0\}$ by
\[d(\gamma )=\text{dist}_{\tilde{\Phi}(r(\gamma))}(\gamma, r(\gamma)\tilde{\calq}W_0)\]
for $\gamma \in \tilde{\calq}$, where for $w\in W$, $\text{dist}_{\tilde{\Phi}(w)}$ is the graph metric on $\tilde{\Phi}(w)$.
For $n\in \N$, we define a map $f_n\colon \{ d=n \} \to \{ d=n-1\}$ as in Section \ref{sec-treeing}.
That is, for $\gamma \in \{ d=n\}$, we choose the minimal $k\in \N$ such that $s(\gamma)\in r(\tilde{\varphi}_k)$ and $\gamma \tilde{\varphi}_k\in \{ d=n-1\}$, and define $f_n(\gamma)= \gamma \tilde{\varphi}_k$.
Finally we define the sets
\[\tilde{\Phi}_{0, +}=\{ \, \gamma^{-1}f_n(\gamma)\mid \gamma \in \{ d=n\},\, n\in \N \, \}\ \ \text{and}\ \ \tilde{\Phi}_0=\tilde{\Phi}_{0, +}\cup \tilde{\Phi}_{0, +}^{-1}.\]
To compute the cost of $\tilde{\Phi}_1\coloneqq \tilde{\Phi}\setminus \tilde{\Phi}_0$, we need the following:

\begin{lem}\label{lem-wn}
With the above notation, for all $w\in W$ and $n\in \N \cup \{ 0\}$, we have
\[w\tilde{\calq}\cap \{ d=n \} =w\tilde{\calq}W_n\cup w\tilde{\calq}W_{-n}.\]
Moreover the following assertions hold:
\begin{enumerate}
\item[(i)] For all $n\in \N$ and $w\in W_n$, the set $\tilde{\Phi}_+\cap w\tilde{\calq}$ consists of exactly $q$ elements, and exactly one of them is an element of $\tilde{\Phi}_{0, +}$.
\item[(ii)] For all $n\in \N$ and $w\in W_{-n}$, the set $\tilde{\Phi}_+^{-1}\cap w\tilde{\calq}$ consists of exactly $p$ elements, and exactly one of them is an element of $\tilde{\Phi}_{0, +}$.
\end{enumerate}
\end{lem}

\begin{proof}
Pick $\gamma =(h, (l, \Delta(h)))\in \tilde{\calq}$ and let $n$ be the integer such that $(q/p)^n=\Delta(h)^{-1}l$.
The equation in the lemma is equivalent to saying that $d(\gamma)=|n|$, which we will verify by induction on $|n|$.
Note that $\Delta(h)^{-1}l$ is the $\tilde{L}^0$-coordinate of $s(\gamma)$.
If $n=0$, then the desired equation holds.

Suppose $n>0$.
Let $k$ be the minimal odd number such that $s(h)\in r(\varphi_k)$. 
Such $k$ exists because if we pick $y\in Z_+$ with $\sigma_+(y)=s(h)$, then $\bar{\theta}(y, t)\in \Phi_+$ and hence $\bar{\theta}(y, t)\in \varphi_k$ for some odd $k$ with $r(\bar{\theta}(y, t))=s(h)$.
The product $\gamma \tilde{\varphi}_k$ is defined, and its $\tilde{L}$-coordinate is
\[(l, \Delta(h))(\Delta(h)^{-1}l, q/p)=(l, \Delta(h)q/p).\]
By the induction hypothesis, $d(\gamma \tilde{\varphi}_k)=n-1$.
Hence $d(\gamma)\leq n$.

Put $m= d(\gamma)$.
There exist odd $k_1,\ldots, k_m$ and $\ve_1,\ldots, \ve_m\in \{ \pm 1\}$ such that the product $\gamma \tilde{\varphi}_{k_1}^{\ve_1}\cdots \tilde{\varphi}_{k_m}^{\ve_m}$ is defined and is in $r(\gamma)\tilde{\calq}W_0$.
The product has the source whose coordinate in $\tilde{L}^0$ is
\[[\Delta(h)(q/p)^{\ve_1+\cdots +\ve_m}]^{-1}l=(q/p)^{n-(\ve_1+\cdots +\ve_m)},\]
which has to be $1$.
Therefore $m\geq n$ and $m=n$.
Moreover the minimality of $k$ implies that $f_n(\gamma)=\gamma \tilde{\varphi}_k$ and hence $s(\gamma)\tilde{\varphi}_k\in \tilde{\Phi}_{0, +}$.

The proof for the case of $n<0$ is similar, given as follows:
Choose the minimal odd $k$ such that $s(h)\in s(\varphi_k)$.
Then by the induction hypothesis, we have $d(\gamma \tilde{\varphi}_k^{-1})=-n-1$, and the equation $d(\gamma)=-n$ follows similarly to the above.
Moreover $f_n(\gamma)=\gamma \tilde{\varphi}_k^{-1}$ and hence $s(\gamma)\tilde{\varphi}_k^{-1}\in \tilde{\Phi}_{0, +}$.
In particular the equation in the lemma was proved.

To prove assertion (i), pick $n\in \N$, $w\in W_n$ and an arbitrary $\gamma =(h, (l, \Delta(h)))\in \tilde{\calq}w$.
Then $\Delta(h)^{-1}l=(q/p)^n$.
Let $k$ be the minimal odd number such that $s(h)\in r(\varphi_k)$.
Then the conclusion in the third paragraph of this proof says that $s(\gamma)\tilde{\varphi}_k=w\tilde{\varphi}_k\in \tilde{\Phi}_{0, +}$.
Assertion (i) now follows from Remark \ref{rem-unique-0+} and Lemma \ref{lem-tilde-Phi-treeing} (i).

Assertion (ii) is verified similarly as follows:
Pick $n\in \N$, $w\in W_{-n}$ and an arbitrary $\gamma =(h, (l, \Delta(h)))\in \tilde{\calq}w$.
Then $\Delta(h)^{-1}l=(q/p)^{-n}$.
Let $k$ be the minimal odd number such that $s(h)\in s(\varphi_k)$.
By the conclusion in the fourth paragraph, $s(\gamma)\tilde{\varphi}_k^{-1}=w\tilde{\varphi}_k^{-1}\in \tilde{\Phi}_{0, +}$, and assertion (ii) follows from Remark \ref{rem-unique-0+} and Lemma \ref{lem-tilde-Phi-treeing} (ii).
\end{proof}

Let $n\in \N$.
For $w\in W_n$, let $\delta_w\in w\tilde{\calq}\cap \tilde{\Phi}_{0, +}$ be the unique element in Lemma \ref{lem-wn} (i).
Similarly for $w\in W_{-n}$, let $\delta_w\in w\tilde{\calq}\cap \tilde{\Phi}_{0, +}$ be the unique element in Lemma \ref{lem-wn} (ii).
By Remark \ref{rem-unique-0+}, the equation $\tilde{\Phi}_{0, +}=\{ \, \delta_w\mid w\in W\setminus W_0\, \}$ holds.
It follows from Lemma \ref{lem-wn} (i), (ii) that the cost of $\tilde{\Phi}_1=\tilde{\Phi}\setminus \tilde{\Phi}_0$ is 
\begin{align*}
C(\tilde{\Phi}_1)&=(q-1)\sum_{n=1}^\infty \tilde{\zeta}(W_n)+(p-1)\sum_{n=1}^\infty \tilde{\zeta}(W_{-n})\\
&=(q-1)\sum_{n=1}^\infty (q/p)^{-n}+(p-1)\sum_{n=1}^\infty (q/p)^n,
\end{align*}
which is $\infty$ if $p\neq 1$ and $q\neq 1$.
Note that otherwise the right hand side is $1$ (since we are assuming $p\neq q$).

Following Section \ref{sec-treeing}, we define a map $f\colon \tilde{\calq} \to \tilde{\calq}W_0$ by
\[f(\gamma)=\begin{cases}
\gamma & \text{if}\ \gamma \in \tilde{\calq}W_0,\\
(f_1\circ f_2\circ \cdots \circ f_n)(\gamma) & \text{if}\ \gamma \in \{  d=n \} \ \text{and}\ n\in \N.
\end{cases}\] 
For $\gamma \in \tilde{\Phi}_1$, we define $J(\gamma)\in \tilde{\calq}|_{W_0}$ by $J(\gamma)=f(r(\gamma))^{-1}f(\gamma)$.
By Propositions \ref{prop-theta-tree} and \ref{prop-theta-cost}, the set $\Theta \coloneqq J(\tilde{\Phi}_1)$ is a treeing of $\tilde{\calq}|_{W_0}$ and its cost is $C(\Theta)=C(\tilde{\Phi}_1)$.
As a conclusion, we obtain the following:

\begin{thm}\label{thm-cost-infty}
Let $G$ be the HNN extension
\[G=\langle \, E, \, t\mid \forall a\in E_-\ \ tat^{-1}=\tau(a)\, \rangle,\]
where $E$ is a countable group and $\tau \colon E_-\to E_+$ is an isomorphism between finite-index normal subgroups $E_-$, $E_+$ of $E$.
We set $p=[E: E_-]$ and $q=[E: E_+]$ and suppose that $p\neq 1$, $q\neq 1$ and $p\neq q$.
Let $\bm{m}\colon G\to \Q^*_+$ be the modular homomorphism associated to $E$ and set $N=\ker \bm{m}$.

Let $G\c (X, \mu)$ be a p.m.p.\ action and suppose that there exist $E$-equivariant Borel maps $X\to E/E_-$ and $X\to E/E_+$.
We set $\caln =X\rtimes N$ and $\cale =X\rtimes E$.
Then $\cale \vartriangleleft \caln$ and the quotient groupoid of $(\caln, \mu)$ by $\cale$ is p.m.p.\ and admits a treeing of cost $\infty$.
\end{thm}

\begin{rem}
If $p=q$, then $(\calq, \zeta)$ and $(\tilde{\calq}, \tilde{\zeta})$ are identified and p.m.p., and the treeing $\Phi$ of $\calq$ has the cost $C(\Phi)=p=q$ by Lemma \ref{lem-tilde-Phi-treeing}.
\end{rem}


\section{Splitting of groupoid-extensions}\label{sec-splitting}

With the notation in Theorem \ref{thm-cost-infty}, assuming that $E$ is finitely generated, free abelian, we show that $N$ is orbit equivalent to $F_\infty \times \Z$.
In Subsection \ref{subsec-general}, we show this splitting result in a general framework in order to clarify what is necessary for the splitting.
In Subsection \ref{subsec-app}, we apply it for proving Theorem \ref{thm-main}.

\subsection{A general splitting result}\label{subsec-general}


For a Borel action $G\c X$ of a countable group on a standard Borel space, we denote its orbit equivalence relation by
\[\calr(G\c X)=\{ \, (x, g^{-1}x)\mid x\in X,\, g\in G\, \}.\]

\begin{lem}\label{lem-split-hyp}
Let $H$ be a countable group and $E$ a normal subgroup of $H$.
Let $H\c X$ be a Borel action on a standard Borel space such that $E$ acts on $X$ trivially.
Set $\calh =X\rtimes H$ and $\calr =\calr (H/E \c X)$.
Let $q\colon \calh \to \calr$ be the quotient map defined by $(x, h)\mapsto (x, h^{-1}x)$ for $x\in X$ and $h\in H$.
Suppose that $\calr$ is hyperfinite.
Then the following exact sequence of groupoids splits:
\[1\to X\times E\to \calh \stackrel{q}{\to} \calr \to 1.\]
Namely, there exists a Borel homomorphism $\sigma \colon \calr \to \calh$ which is a section of $q$.
\end{lem}

\begin{proof}
Write $\calr =\bigcup_n \calr_n$ as the increasing union of finite subequivalence relations $\calr_n$.
We inductively find an increasing sequence of subgroupoids $\calh_n<\calh$ which are sections of $\calr_n$.

Choose a transversal $F_n\subset X$ of $\calr_n$ such that $F_{n+1}\subset F_n$, where a \textit{transversal} of $\calr_n$ is a Borel subset of $X$ which meets each $\calr_n$-equivalence class in exactly one point.
Write $F_1\calr_1=\bigsqcup_k\phi_k$ as the union of countably many $\phi_k\in \llbracket \calr_1\rrbracket$.
Choose $\psi_k\in \llbracket \calh \rrbracket$ such that $q(\psi_k)=\phi_k$, and let $\calh_1$ be the subgroupoid generated by $(\psi_k)_k$.
Then $\calh_1$ is isomorphic to $\calr_1$ via $q$.

Suppose that we have constructed subgroupoids $\calh_1<\cdots < \calh_n$.
Write $F_{n+1}\calr_{n+1}F_n=\bigsqcup_k\phi_k'$ as the union of countably many $\phi_k'\in \llbracket \calr_{n+1} \rrbracket$.
Choose $\psi_k'\in \llbracket \calh \rrbracket$ such that $q(\psi_k')=\phi_k'$, and let $\calh_{n+1}$ be the subgroupoid generated by $(\psi_k')_k$ and $\calh_n$.
Then $\calh_{n+1}$ is isomorphic to $\calr_{n+1}$ via $q$.

The union $\bigcup_n \calh_n$ is a subgroupoid of $\calh$ isomorphic to $\calr$ via $q$.
\end{proof}

In the rest of this subsection, we keep the following:

\begin{notation}\label{not-split}
Let $\varphi \colon N\to H$ be a surjective homomorphism between countable groups.
Let $E<N$ be a subgroup and suppose that $\varphi$ is injective on $E$ and the image $\varphi(E)$ is central in $H$.
We identify $E$ with $\varphi(E)$ via $\varphi$.
Suppose that $H/E$ is amenable.

Let $H\c (X, \mu)$ be a p.m.p.\ action such that $E$ acts on $X$ trivially and $H/E$ acts on $X$ freely.
We set $\calh =X\rtimes H$ and $\calr =\calr (H/E\c X)$.
By Lemma \ref{lem-split-hyp}, there exists a Borel homomorphism $\sigma \colon \calr \to \calh$ which is a section of the quotient map $\calh \to \calr$.
We define $\calr \times E$ as the direct product of the two groupoids $\calr$ and $E$.
Then $\sigma$ gives rise to the isomorphism $\calr \times E \simeq \calh$ given by $(\gamma, a)\mapsto \sigma(\gamma)(s(\gamma), a)$ for $\gamma \in \calr$ and $a\in E$.
This is indeed a homomorphism, because $E$ is central in $H$ and hence for all $\delta \in \calh$ and $a\in E$, we have $(r(\delta), a)\delta =\delta (s(\delta), a)$.

We pick an ergodic free p.m.p.\ action $E\c (Y, \nu)$.
Let $\calr \times E$ act on $X\times Y$ componentwise:
$(\gamma, a)(s(\gamma), y)=(r(\gamma), ay)$ for $\gamma \in \calr$, $a\in E$ and $y\in Y$.
Via the isomorphism $\calr \times E\simeq \calh$, let $\calh$ act on $X\times Y$.
Using this action, we define an action $H\c X\times Y$ by $h(x, y)=(hx, h)(x, y)$ for $h\in H$, $x\in X$ and $y\in Y$.
Note that $E$ acts on $X\times Y$ by $a(x, y)=(x, ay)$ for $a\in E$.
Finally let $N$ act on $X\times Y$ via the homomorphism $\varphi \colon N\to H$, and set $\caln =(X\times Y)\rtimes N$ and $\cale =(X\times Y)\rtimes E$.
\end{notation}

\begin{ex}
We give an example of the homomorphism $\varphi \colon N\to H$ in Notation \ref{not-split}.
Let $G=\bs(2, 3)=\langle \, a,\, t\mid ta^2t^{-1}=a^3\, \rangle$ and $E=\langle a\rangle$.
Then $E$ is quasi-normal in $G$ and we have the associated modular homomorphism $\bm{m}\colon G\to \Q^*_+$.
Set $N=\ker \bm{m}$.
Let $\Q^*_+\ltimes \Q$ be the semidirect product group where $\Q^*_+$ acts on $\Q$ by multiplication.
We have the homomorphism $\varphi \colon G\to \Q^*_+\ltimes \Q$ such that $\varphi(a)=1\in \Q$ and $\varphi(t)=3/2\in \Q^*_+$.
We have $\varphi(N)=\Z[1/6]<\Q$ and set $H=\varphi(N)$.
Then $\varphi$ is injective on $E$, $\varphi(E)$ is central in $H$, and $H/\varphi(E)$ is abelian and hence amenable. 
\end{ex}

\begin{lem}\label{lem-split-ne}
With Notation \ref{not-split}, suppose that $\cale$ is normal in $(\caln, \mu \times \nu)$ and denote the quotient groupoid by $(\caln /\cale, \zeta)$.
Then we have the isomorphism
\[(\caln, \mu \times \nu)\simeq (\caln/\cale \times (Y\rtimes E), \zeta \times \nu),\]
where $\caln/\cale \times (Y\rtimes E)$ is the direct product of the two groupoids $\caln/\cale$ and $Y\rtimes E$.
\end{lem}

\begin{proof}
Let $\theta \colon \caln \to \caln /\cale$ be the quotient homomorphism.
Since the action $E\c (Y, \nu)$ is ergodic, the unit space $(\caln /\cale)^0$ is identified with $X$.
We identify the groupoid $Y\rtimes E$ with the orbit equivalence relation $\calr(E\c Y)$ under the map $(y, a)\mapsto (y, a^{-1}y)$ for $y\in Y$ and $a\in E$.
We define a homomorphism $\pi \colon \caln \to Y\rtimes E$ by
\[\pi((x, y), g)=(p_Y(x, y), p_Y(g^{-1}(x, y)))\]
for $x\in X$, $y\in Y$ and $g\in N$, where $p_Y\colon X\times Y\to Y$ is the projection.

We define a homomorphism
\[F\colon \caln \to \caln /\cale \times (Y\rtimes E)\]
by $F(\gamma)=(\theta(\gamma), \pi(\gamma))$ for $\gamma \in \caln$.
We show that $F$ is an isomorphism.
To verify injectivity, pick $\gamma \in \caln$ such that $F(\gamma)$ is a unit.
Then $\theta(\gamma)$ is a unit and hence $\gamma \in \ker \theta =\cale$.
We can write $\gamma =((x, y), a)$ for some $a\in E$ and $(x, y)\in X\times Y$.
We then have $\pi(\gamma)=(y, a^{-1}y)$.
Since $\pi(\gamma)$ is a unit and the action $E\c Y$ is free, we have $a=e$ and hence $\gamma$ is a unit.

To verify surjectivity, pick $(\gamma, (y, a^{-1}y))\in \caln /\cale \times (Y\rtimes E)$ with $\gamma \in \caln /\cale$, $y\in Y$ and $a\in E$.
By class-surjectivity of $\theta$, there exists $g\in (r(\gamma), a^{-1}y)\caln$ such that $\theta(g)=\gamma$.
Define $b\in E$ by $\pi(g)=(a^{-1}y, b)$.
Note that we have $s(g)=(s(\gamma), s(\pi(g)))$ by the definition of $\pi$, and $s(\pi(g))=b^{-1}a^{-1}y$.
The product
\[((r(\gamma), y), a)g((s(\gamma), b^{-1}a^{-1}y), b^{-1})\in \caln\]
is therefore defined.
Its image under $\theta$ is $\theta(g)=\gamma$, and the image under $\pi$ is
\[(y, a)(a^{-1}y, b)(b^{-1}a^{-1}y, b^{-1})=(y, a).\]
Thus surjectivity of $F$ follows.
\end{proof}

\begin{lem}\label{lem-oe-fne}
With Notation \ref{not-split}, we again suppose that $\cale$ is normal in $(\caln, \mu \times \nu)$ and denote the quotient groupoid by $(\caln /\cale, \zeta)$.
We also suppose that $(\caln /\cale, \zeta)$ is ergodic and p.m.p.\ and admits a treeing of cost $n\in \N \cup \{ \infty \}$.
Then $N$ is orbit equivalent to $F_n\times E$, where $F_n$ is the free group of rank $n$. 
\end{lem}

\begin{proof}
We refer the reader to \cite[Section 3]{btd} for the terminology on p.m.p.\ actions of discrete p.m.p.\ groupoids employed in this proof.
Put $\calm =\caln /\cale$.
Let $\alpha \colon (\calm, \zeta) \c (\Sigma, \omega)$ be the standard Bernoulli action of $(\calm, \zeta)$ with nontrivial base (\cite[Definition 3.1]{btd}), where $(\Sigma, \omega)$ is a fibered probability space over $(\calm^0, \zeta)$ with the projection $p\colon \Sigma \to \calm^0$.
Let $(\Sigma \rtimes \calm, \omega)$ be the translation groupoid associated to the action $\alpha$, defined as follows (\cite[Definition 3.5]{btd}):
As a set, $\Sigma \rtimes \calm$ is the fibered product $\Sigma \times_{\calm^0}\calm$ with respect to $p$ and the range map of $\calm$.
The unit space of $\Sigma \rtimes \calm$ is $\Sigma$.
The range and source maps are defined by $r(u, \gamma)=u$ and $s(u, \gamma)=\alpha(\gamma)^{-1}u$ for $(u, \gamma)\in \Sigma \rtimes \calm$, respectively.
The product and inverse are defined by $(u, \gamma)(\alpha(\gamma)^{-1}u, \delta)=(u, \gamma \delta)$ and $(u, \gamma)^{-1}=(\alpha(\gamma)^{-1}u, \gamma^{-1})$, respectively.

Let $\Phi$ be a treeing of $(\calm, \zeta)$ of cost $n$ which exists by our assumption.
Then we obtain the treeing
\[\tilde{\Phi}=\{ \, (u, \gamma)\mid u\in p^{-1}(r(\gamma)),\, \gamma \in \Phi \, \}\]
of $(\Sigma \rtimes \calm, \omega)$ of the same cost $n$.
In fact for all $u\in \Sigma$, the map $u(\Sigma \rtimes \calm)\to p(u)\calm$ defined by $(u, \gamma)\mapsto \gamma$ induces a graph isomorphism from $\tilde{\Phi}(u)$ onto $\Phi(p(u))$.

Since $(\calm, \zeta)$ is ergodic, it follows from \cite[Lemma 3.27]{btd} that $(\Sigma \rtimes \calm, \omega)$ is ergodic.
By \cite[Proposition 3.29]{btd}, $(\Sigma \rtimes \calm, \omega)$ is principal.
It thus turns out from \cite[Lemmas 4.1 and 4.2]{hjorth} that $(\Sigma \rtimes \calm, \omega)$ arises from some free p.m.p.\ action $F_n\c (\Sigma, \omega)$.

Let $\calm \times (Y\rtimes E)$ act on $\Sigma \times Y$ componentwise:
$(\gamma, (y, a))(u, a^{-1}y)=(\alpha(\gamma)u, y)$ for $\gamma \in \calm$, $u\in p^{-1}(s(\gamma))$, $y\in Y$ and $a\in E$.
Via the isomorphism $\caln \simeq \calm \times (Y\rtimes E)$ in Lemma \ref{lem-split-ne}, let $\caln$ act on $\Sigma \times Y$ and let the group $N$ act on $\Sigma \times Y$ by $gw=(gq(w), g)w$ for $g\in N$, $w\in \Sigma \times Y$, where $q\colon \Sigma \times Y\to \caln^0$ is the projection that makes $\Sigma \times Y$ into the fibered space on which $\caln$ acts.
The translation groupoid $(\Sigma \times Y )\rtimes N$ is isomorphic to
\[(\Sigma \times Y )\rtimes \caln \simeq (\Sigma \rtimes \calm) \times (Y\rtimes E).\]
The right hand side (endowed with the measure $\omega \times \nu$) is the translation groupoid arising from a free p.m.p.\ action of $F_n\times E$.
\end{proof}


\subsection{An application}\label{subsec-app}

Recall the notation in Theorem \ref{thm-cost-infty}.
Let $G$ be the HNN extension
\[G=\langle \, E, \, t\mid \forall a\in E_-\ \ tat^{-1}=\tau(a)\, \rangle,\]
where $E$ is a countable group and $\tau \colon E_-\to E_+$ is an isomorphism between finite-index normal subgroups $E_-$, $E_+$ of $E$.
We set $p=[E: E_-]$ and $q=[E: E_+]$ and suppose that $p\neq 1$, $q\neq 1$ and $p\neq q$.
Let $\bm{m}\colon G\to \Q^*_+$ be the modular homomorphism associated to $E$ and set $N=\ker \bm{m}$.

We further assume that $E$ is finitely generated, free abelian and identify $E$ with $\Z^\nu$.
Then the isomorphism $\tau \colon E_-\to E_+$ is given by an element of $\mathit{GL}_\nu(\Q)$ and hence extends to an automorphism of $\Q^\nu$, which we denote by the same symbol $\tau$.
We define the subgroup $H\coloneqq \bigvee_{n\in \Z}\tau^n(E)<\Q^\nu$.

Let $\mathit{GL}_\nu(\Q)\ltimes \Q^\nu$ be the semidirect product group where $\mathit{GL}_\nu(\Q)$ acts on $\Q^\nu$ by linear transformations.
We have the homomorphism $\varphi \colon G\to \mathit{GL}_\nu(\Q)\ltimes \Q^\nu$ that embeds $E=\Z^\nu$ into $\Q^\nu$ and sends $t$ to $\tau \in \mathit{GL}_\nu(\Q)$.
Then $\varphi(t^nat^{-n})=\tau^n(a)$ for all $a\in E$ and $n\in \Z$, and in particular $H=\varphi(N)$.

Pick a p.m.p.\ action $G\c (X, \mu)$ such that
\begin{itemize}
\item there exist $E$-equivariant Borel maps $X\to E/E_-$ and $X\to E/E_+$, and
\item the restriction $N\c (X, \mu)$ is ergodic
\end{itemize}
(e.g., the action of $G$ coinduced from the diagonal action $E\c E/E_-\times E/E_+$ fulfills these properties).
We set $\caln =X\rtimes N$ and $\cale =X\rtimes E$.
It follows from Theorem \ref{thm-cost-infty} that $\cale \vartriangleleft \caln$ and the quotient groupoid of $(\caln, \mu)$ by $\cale$ is p.m.p.\ and admits a treeing of cost $\infty$.
By Lemma \ref{lem-oe-fne}, $N$ is orbit equivalent to $F_\infty \times E$.
Thus we obtain Theorem \ref{thm-main}.








\end{document}